\documentclass[12pt,twoside,a4paper]{amsart}

\usepackage[margin=2.5cm]{geometry}

\usepackage{amssymb,mathtools}
\usepackage{xcolor}

\newtheorem{thm}{Theorem}[section]
\newtheorem*{thmA}{Theorem A}
\newtheorem*{thmB}{Theorem B}
\newtheorem*{thmC}{Theorem C}
\newtheorem*{thmD}{Theorem D}
\newtheorem*{thmE}{Theorem E}

\newtheorem{lem}[thm]{Lemma}
\newtheorem{prop}[thm]{Proposition}
\newtheorem{cor}[thm]{Corollary}
\theoremstyle{definition}
\newtheorem{defn}[thm]{Definition}
\theoremstyle{remark}
\newtheorem{rem}[thm]{Remark}
\newtheorem{example}[thm]{Example}

\numberwithin{equation}{section}
\allowdisplaybreaks

\newcommand{\Hn}{\mathbb{H}^n}

\newcommand{\C}{\mathbb{C}}
\newcommand{\R}{\mathbb{R}}
\newcommand{\Z}{\mathbb{Z}}
\newcommand{\N}{\mathbb{N}}

\newcommand{\Rplus}{\mathbb{R}^+}

\renewcommand{\Im}{\operatorname{Im}}

\newcommand{\Lie}[1]{\mathfrak{#1}}

\newcommand{\trace}{\operatorname{trace}}
\newcommand{\ad}{\operatorname{ad}}
\newcommand{\supp}{\operatorname{supp}}
\newcommand{\loc}{\mathrm{loc}}

\newcommand{\afterbar}{^-}
\newcommand{\afterarrow}{^{\to}}
\newcommand{\aftertilde}{\tilde{\phantom{n}}}
\newcommand{\sfrac}[2]{#1 / #2}
\newcommand{\wrt}[1]{\,\mathrm{d} #1}

\newcommand{\BL}{\operatorname{\mathbf{BL}}}
\newcommand{\bsf}{\boldsymbol{f}}
\newcommand{\bss}{\boldsymbol{\sigma}}
\newcommand{\bsds}{\boldsymbol{\dot\sigma}}
\newcommand{\bsDs}{\boldsymbol{\mathrm{d}\sigma}}
\newcommand{\bsts}{\boldsymbol{\tilde\sigma}}
\newcommand{\bsrho}{\boldsymbol{\rho}}
\newcommand{\bsphi}{\boldsymbol{\phi}}
\newcommand{\bspsi}{\boldsymbol{\psi}}
\newcommand{\bspi}{\boldsymbol{\pi}}

\newcommand{\Polytope}{\boldsymbol{P}}
\newcommand{\bsp}{\boldsymbol{p}}
\newcommand{\bstp}{\boldsymbol{\tilde{p}}}
\newcommand{\bsq}{\boldsymbol{q}}
\newcommand{\bsg}{\boldsymbol{\Lie{g}}}
\newcommand{\bsG}{\boldsymbol{G}}
\newcommand{\bsGc}{\boldsymbol{G^\circ}}
\newcommand{\bsH}{\boldsymbol{H}}
\newcommand{\bsU}{\boldsymbol{U}}
\newcommand{\bsV}{\boldsymbol{V}}

\newcommand{\norm}[1]{\left\Vert #1 \right\Vert}

\newcommand{\lpar}{\left( }
\newcommand{\rpar}{\right) }
\newcommand{\bigglpar}{\biggl(}
\newcommand{\biggrpar}{\biggr)}

\newcommand{\labs}{\left\vert}
\newcommand{\rabs}{\right\vert}
\newcommand{\biglabs}{\bigl\vert}
\newcommand{\bigrabs}{\bigr\vert}
\newcommand{\Biglabs}{\Bigl\vert}
\newcommand{\Bigrabs}{\Bigr\vert}
\newcommand{\bigglabs}{\biggl\vert}
\newcommand{\biggrabs}{\biggr\vert}

\begin{document}

\thanks{
The authors thank Jonathan Bennett for useful discussions about Brascamp--Lieb inequalities. \\  
The first and second authors were supported by the Australian Research Council through grants DP220100285 and DP260101083. 
The third author was supported in part by NSTC through grant 111-2115-M-002-010-MY5 and in part by a Macquarie University co-tutelle PhD scholarship.}

\title[Brascamp--Lieb inequalities]{The Brascamp--Lieb inequality on compact Lie groups \\ and its extinction on homogeneous Lie groups}
\author{Michael G. Cowling} 
\address{School of Mathematics and Statistics\\ University of New South Wales\\ Sydney NSW 2052\\ Australia}
\email{m.cowling@unsw.edu.au}
\author{Ji Li} 
\address{School of Mathematical and Physical Sciences\\ Macquarie University\\ NSW 2109\\ Australia.}
\email{ji.li@mq.edu.au}

\author{Chong-Wei Liang}
\address{Department of Mathematics, National Taiwan University, Taiwan.}
\email{d10221001@ntu.edu.tw}

\date{}

\begin{abstract}
We study the Brascamp--Lieb inequalities on locally compact nonabelian groups and the Brascamp--Lieb constants $\mathbf{BL}(G, \boldsymbol{\sigma}, \boldsymbol{p})$ associated to a Brascamp--Lieb datum: locally compact groups $G$ and $G_j$, a family of homomorphisms $\sigma_j\colon  G \to G_j$ and Lebesgue indices $p_j$.
We focus on homogeneous Lie groups and compact Lie groups.
For homogeneous Lie groups $G$, we show that the constant $\mathbf{BL}(G, \boldsymbol{\sigma}, \boldsymbol{p})$  is equal to the constant $\mathbf{BL}(\mathfrak{g}, \boldsymbol{\mathrm{d}\sigma}, \boldsymbol{p})$, where $\mathfrak{g}$ is the Lie algebra of $G$ and $\mathrm{d}\sigma_j$ is the differential of $\sigma_j$.
For Heisenberg-like groups $G$, we show that the only inequalities that can occur are multilinear Hölder inequalities.
For compact Lie groups, we find necessary and sufficient conditions for finiteness of the constant $\mathbf{BL}(G, \boldsymbol{\sigma}, \boldsymbol{p})$ in terms of $\boldsymbol{\sigma}$ and $\boldsymbol{p}$ and find an explicit expression for the constant, similar to those found by Bennett and Jeong in the abelian case.
\end{abstract}

\maketitle

\section{Introduction}
The Brascamp--Lieb inequalities (also known as Hölder--Brascamp--Lieb inequalities), introduced in \cite{BL76}, are of considerable interest across mathematics; see, for example, \cite{BB21, BCCT08, Zh22}  for a discussion of their influence, particularly in Fourier analysis. 
They provide Lebesgue space bounds for rather broad classes of multilinear forms that arise in a range of settings, from harmonic analysis and dispersive PDE, to additive combinatorics, convex geometry and theoretical computer science (see \cite{CDKSY24}). 
Originally they were phrased in the context of euclidean spaces, but more recent versions consider discrete abelian groups and finite-dimensional compact abelian groups, as well as compact homogeneous spaces and the Heisenberg group. 
For more background on these inequalities, see below.

The main aim of this paper is to describe the Brascamp--Lieb inequalities on more general locally compact groups, by which we mean locally compact Hausdorff topological groups, with a particular focus on nilpotent and compact Lie groups. 
Here duality is a less useful tool, as the dual of a nonabelian group is not a group.
Our message is that the absence of commutativity means that there are many fewer potential inequalities.

We now define Brascamp--Lieb inequalities, outline some of their history, and present our main results.

\begin{defn}
A \emph{Brascamp--Lieb datum} is a triple $(G, \bss,\bsp)$, where $G$ is a locally compact group, $\bss$ is a $J$-tuple of homomorphisms $\sigma_j\colon G\to  G_j$ of locally compact groups, and $\bsp$ is a $J$-tuple of indices $(p_1,\dots ,p_J)$, each in $[1,\infty]$. 
A Brascamp--Lieb input $\bsf$ is a $J$-tuple of functions $f_j$ in $L^{p_j}(G_j)$.
\end{defn}

We \emph{always} suppose that homomorphisms are continuous.
The number of homomorphisms, that is, $J$, is not particularly important, and we often omit to mention it.
We write $\bsG$ for the product group $G_1 \times \dots \times G_J$; we may consider $\bss$ as the homomorphism of $G$ into $\bsG$ whose components are the $\sigma_j$. 

\begin{defn}
A \emph{Brascamp--Lieb inequality} is an inequality of the form
\begin{equation}\label{eq:usual}
\bigglabs \int_{G} \prod_j   (f_j \circ \sigma_j)(x) \wrt x \biggrabs  
\leq C \prod_{j} \left\Vert f_j \right\Vert_{L^{p_j}(G_j)}
\qquad\forall f_j \in C_c(G_j) ,
\end{equation}
where $C$ is a constant.
The product function on the left hand side of this inequality is assumed to be integrable.
The smallest possible value of $C$ in \eqref{eq:usual} is called the \textit{Brascamp--Lieb constant} and denoted by $\BL(G, \bss, \bsp)$; we write $\BL(G,\bss,\bsp) = \infty$ if no such inequality holds. 
\end{defn}

For finite groups $G$ and $\bsG$, the Brascamp--Lieb constant is finite.

The Brascamp--Lieb constant $\BL(G, \bss, \bsp)$ clearly depends on the choices of the Haar measures on $G$ and the $G_j$ that implicitly appear in \eqref{eq:usual}.
However, finiteness of the constants does not, since changing the Haar measure by a factor changes the left hand side of \eqref{eq:usual}  by the same factor and the right hand side  by a power of the factor.

We now present a very brief history of the Brascamp--Lieb inequalities, which  generalise classical inequalities such as Hölder's inequality, Young's convolution inequality and the Loomis--Whitney inequalities.
The modern story begins with Brascamp and Lieb \cite{BL76}, who studied generalisations of Young's inequality on \emph{euclidean spaces}.
General finiteness results when $G$ and $\bsG$ are euclidean spaces were found independently by Barthe \cite{Barthe98} and Carlen, Lieb and Loss \cite{CLL04}, who treated the case where each $G_j$ is a copy of $\R$.
Then Bennett, Carbery, Christ and Tao \cite{BCCT08} found necessary and sufficient conditions for the finiteness of $\BL(G,\bss,\bsp)$ when the $G_j$ are euclidean spaces of arbitrary finite dimension, which we now recall.
\begin{defn}\label{def:BCCT}
Let $(G, \bss,\bsp)$ be a Brascamp--Lieb datum, where $G$ and $\bsG$ are (real) vector spaces and $\bss$ is a linear mapping.
The \emph{BCCT conditions} are the scaling equality
\begin{align}\label{eq:BCCT-1}
\dim(G) &=\sum^J_{j=1}\frac{1}{p_j} \dim(G_j)           \\
\noalign{\noindent and the dimension inequalities}
\label{eq:BCCT-2}
\dim(H)  &\leq \sum^J_{j=1}\frac{1}{p_j} \dim(\sigma_{j}(H))
\qquad \forall H \in \Lie{S} (G) ,       
\end{align}
where $\Lie{S}(G)$ is the collection of all vector subspaces of $G$.
\end{defn}
For euclidean spaces, it is known how to compute the constants when they are finite, at least in principle, thanks to a fundamental result of Lieb \cite{Lieb90}, which states that it suffices to test on centred gaussian functions $f_j$.

For \emph{finitely generated discrete abelian groups} $G$ and $\bsG$  it was shown by Bennett, Carbery, Christ and Tao \cite{BCCT10} that the finiteness of $\BL(G,\bss,\bsp)$ is equivalent to the rank condition, that is, \eqref{eq:BCCT-2} where dimension is replaced by rank and $H$ varies over all subgroups of $G$.
A discrete analogue of Lieb's theorem was later obtained by Christ \cite{Chr13}, allowing the constant to be computed by testing on indicator functions of finite subgroups of $G$.

Bennett and Jeong reached similar conclusions in \cite[Theorems 4.1 and 5.1]{BJ22} for \emph{compact abelian Lie groups}.
Here a necessary and sufficient criterion for finiteness of the Brascamp--Lieb constant is the family of codimension inequalities
\begin{equation}\label{eq:codimension-condition-group}
\sum_j \frac{1}{p_j} \lpar \dim(G_j)  - \dim( \sigma_j (H) ) \rpar  
\leq  \dim(G) - \dim(H)
\end{equation}
for all closed connected subgroups $H$ of $G$. 
Further, the constant may be evaluated by testing on indicator functions of open subgroups of $G$.

Very recently, Bennett and the first author \cite{BC25} established a structure theorem for the Brascamp--Lieb constant in the setting of \emph{locally compact abelian groups}, which extends and unifies the finiteness characterisations for abelian groups mentioned above. 
They placed particular emphasis on Fourier invariance throughout, reflecting the fundamental Fourier invariance of Brascamp–Lieb multilinear forms in their context.

We mention also the \emph{nonlinear version of the inequality}, in which group homomorphisms are replaced by smooth mappings  \cite{BBBCF20}.
At any given point $x$, the derivative $\mathrm{d}\sigma_j$ of the mapping $\sigma_j$ is a linear mapping of the tangent space $T_xG$ to $T_{\sigma_j(x)}G_j$.
It is shown that, if at some point $x$, the derivatives satisfy the condition
\[
\dim(H)  \leq \sum^J_{j=1}\frac{1}{p_j} \dim(\mathrm{d}\sigma_{j}(H))
\]
for all subspaces $H$ of the tangent space $T_xG$, then a local Brascamp--Lieb inequality holds, that is, an inequality \eqref{eq:usual} with the extra condition that the support of each $f_j$ lies in a sufficiently small neighbourhood of $\sigma_j(x)$.
This result has been used to prove particular cases of Brascamp--Lieb inequalities on Lie groups \cite{CMMP19}.
More precisely, on a unimodular locally compact group $G$, we have Young's inequality for convolution, and iterated versions thereof; here is an example.  

\begin{example}\label{ex:Young}
Let $H$ be a unimodular locally compact group,  $G$ be $H \times H \times H$, $\bsH$ be $H \times H \times H \times H$ and define $\sigma_j \colon  G \to H$ by 
\[
\begin{aligned}
\sigma_1(x_1, x_2,x_3) &= x_1 &&\qquad& \sigma_2(x_1, x_2,x_3) &= x_1^{-1}x_2  \\
\sigma_3(x_1, x_2,x_3) &= x_2^{-1}x_3 &&&  \sigma_4(x_1, x_2,x_3) &= x_3 .
\end{aligned}
\]
Then
\[
\begin{aligned}
\iiint_{\bsG} \prod_j f_j (\sigma_j (y)) \wrt y_1\wrt y_2\wrt y_3 
&= \iint  f_1 * f_2(y_2) f_3(y_2^{-1} y_3) f_4(y_3) \wrt y_2\wrt y_3  \\
&= \int f_1 * f_2 * f_3 (y_3) f_4(y_3) \wrt y_3,
\end{aligned}
\]
and the converse of Hölder's inequality shows that if a Brascamp--Lieb inequality holds for $(G, \bss, \bsp)$, then
\[
\norm{ f_1 * f_2 * f_3 }_{L^{p_4'}(H)} 
\leq C \norm{ f_1 }_{L^{p_1}(H)} \norm{ f_2  }_{L^{p_2}(H)} \norm{ f_3 }_{L^{p_3}(H)}  ,
\]
which is the trilinear form of Young's inequality for convolution.
It may be shown that this holds when $1/p_4' = 1/p_{1} + 1/p_{2} + 1/p_{3} -2$ by interpolating from the cases where two of $p_1$, $p_2$  and $p_3$  are $1$ and the third is $p_4'$.
\end{example}

In this paper, we discuss Brascamp--Lieb inequalities on not necessarily abelian locally compact groups, with a focus on nilpotent and compact Lie groups. 

In the context of locally compact groups, we show in Section \ref{sec:BL-general-groups} that $\BL(G, \bss, \bsp)$ is infinite unless the homomorphism $\bss\colon  G \to \bsG$ is proper, that is, $\bss^{-1}(U)$ is compact whenever $U$ is compact, and unless $\sigma_j$ is open or $p_j = \infty$ for all $j$.
Since our primary interest is in the case where the Brascamp--Lieb constant is finite, we assume that $\bss$ is proper and each $\sigma_j$ is open and surjective.
It then turns out that we may assume that $\ker(\bss)$ is trivial and that $\bss$ is a homeomorphic isomorphism of $G$ with a subgroup $\bss(G)$ of $\bsG$, equipped with the relative topology.
Equivalently, we may assume that our datum is \emph{canonical}, that is, $G$ is a closed subgroup of a product group $\bsG$ and the $\sigma_j$ are the restrictions of $G$ of the canonical projections of $\bsG$ onto the factors $G_j$.
This particular type of configuration was identified in \cite{BBBCF20}.

In Section 4, we consider Brascamp--Lieb constants on homogeneous Lie groups, that is, Lie groups that admit a one-parameter family of automorphic dilations; these are nilpotent (see \cite{Sie86}). 
Here we have two main results.
First, we show  that the Brascamp--Lieb inequality may be linearised.
\begin{thmA}
Let $(G, \bss, \bsG)$ be a canonical Brascamp--Lie datum of homogeneous Lie groups such that $\sigma_j \circ \delta_t = \delta_t \circ \sigma_j$ for all $j$.
Then the derivatives $\mathrm{d}\sigma_j$ of the homomorphisms are linear maps of the corresponding Lie algebras, and 
\[
\BL(G,\bss,\bsp) = \BL(\Lie{g}, \bsDs, \bsg).
\]
\end{thmA}

We also show that noncommutativity is an obstruction to the existence of Brascamp--Lieb inequalities: indeed, for certain ``Heisenberg-like'' groups $G$, the only Brascamp--Lieb inequalities involving homomorphisms $\sigma_j \colon G \to G_j$ are the multilinear Hölder inequalities.  
This result is in a less definitive form, as we do not classify ``Heisenberg-like groups''.

Much of our work concerns the case when the groups $G$ and $\bsG$ are compact Lie groups; in this case, we refer to the datum $(G, \bss, \bsp)$ as a Brascamp--Lieb datum of compact Lie groups.
(It should be noted that the image of a surjective homomorphism of a compact Lie group onto a locally compact group is in fact a compact Lie group, so it suffices to suppose that $G$ is a compact Lie group.)
When the groups $G$ and $\bsG$ are compact, $\bss$ is obviously proper.

We use the structure theory of compact Lie groups to determine the finiteness of the Brascamp--Lieb constant.
We write $G_e$ for the connected component of a topological group $G$.

\begin{thmB}
Suppose that  $(G, \bss,\bsp)$ is a canonical Brascamp--Lieb datum of compact Lie groups. 
Let $\phi_j$ be the restriction of $\sigma_j$ to $G_e$ with codomain $(G_j)_e$. 
Then  $\BL(G, \bss, \bsp)$ is finite if and only if $\BL(G_e, \bsphi, \bsp)$ is finite.
\end{thmB}

This allows us to focus on compact connected Lie groups. 

The appropriate analogue of condition \eqref{eq:codimension-condition-group} enables us to give a finiteness condition for the Brascamp--Lieb constant on a compact Lie group.
To state our condition, we need more notation.
For a compact Lie algebra $\Lie{g}$, we write $\Lie{I}(\Lie{g})$ for the collection of all Lie ideals in $\Lie{g}$, and for a compact Lie group $G$, we write $\Lie{N}(G)$ for the collection of all closed connected normal subgroups of $G$.

\begin{defn}
Let $(G,\bss,\bsp)$ be a Brascamp--Lieb datum of compact connected Lie groups, let $\Lie{g}$ and $\Lie{g}_j$ be the Lie algebras of $G$ and $G_j$, and let $\mathrm{d}\sigma_j\colon  \Lie{g} \to \Lie{g}_j$ be the derivative of the homomorphism $\sigma_j\colon  G \to G_j$.
The codimension conditions are the requirements that 
\begin{equation}\label{eq:codimension-condition}
\sum_j \frac{1}{p_j} \lpar \dim(\mathrm{d}\sigma_j(\Lie{g})))  - \dim( \mathrm{d}\sigma_j (\Lie{h} ) \rpar  
\leq  \dim(\Lie{g}) - \dim(\Lie{h})
\qquad\forall \Lie{h} \in \Lie{I}(\Lie{g});
\end{equation}
or
\begin{equation}\label{eq:codimension-condition-gp}
\sum_j \frac{1}{p_j} \lpar \dim(G_j)  - \dim( \sigma_j (H ) \rpar  
\leq  \dim(G) - \dim(H)
\qquad\forall H \in \Lie{N}(G).
\end{equation}
\end{defn}

Lie theory shows that \eqref{eq:codimension-condition} implies \eqref{eq:codimension-condition-gp}, but not necessarily vice versa; this reverse implication follows from Theorem C.

\begin{thmC}
Suppose that $(G,\bss,\bsp)$ is a canonical datum of compact Lie groups.
If the codimension condition \eqref{eq:codimension-condition-gp} holds, then $\BL(G,\bss,\bsp)$ is finite, while if $\BL(G,\bss,\bsp)$ is finite, then the codimension condition \eqref{eq:codimension-condition} holds.
If all $G_j$ are connected, the Haar measures of $G$ and the $G_j$ are all normalised to be $1$, and $\BL(G,\bss,\bsp)$ is finite, then 
\[
\BL(G, \bss, \bsp)=1.
\]
\end{thmC}

It is noteworthy that all subspaces of a vector space are closed, and the BCCT conditions involve testing on an uncountable number of subspaces;
in tori, there are uncountably many ideals in the Lie algebra, and uncountably many analytic subgroups of the group, but only countably many of these are closed, so testing condition \eqref{eq:codimension-condition-gp} is ``easier'' than testing condition \eqref{eq:codimension-condition}.
Compact semisimple Lie algebras contain finitely many ideals, and compact semisimple Lie groups have finite many closed normal subgroups, so testing either condition  is very straightforward in the semisimple case.

Our next main result splits the abelian and semisimple cases.
For a compact connected Lie group $G$, we write $G'$ for its commutator subgroup and $Z(G)_e$ for the connected component of the centre of $G$ containing the identity $e$, or equivalently, the maximal central torus of $G$. 
As we shall see below, every compact connected Lie group is a product $G' Z(G)_e$ of $G'$, which is semisimple, and of $Z(G)_e$, which is a torus, and $G' \cap Z(G)_e$ is finite.
If $\sigma_j\colon  G \to G_j$ is a surjective open map of compact Lie groups, then $\sigma_j(G') = (\sigma_j(G))'$ and $\sigma_j(Z(G)_e) = Z(\sigma_j(G))_e$.
We write $\phi_j$  for the restriction of $\sigma_j$ to $G'$ with codomain $(\sigma_j(G))'$, and
$\psi_j$ for the restriction of $\sigma_j$ to $Z(G)_e$ with codomain $Z(G_j)_e$.

\begin{thmD}
Suppose that $(G, \bss,\bsp)$ is a canonical Brascamp--Lieb datum of compact Lie groups.
Then $\BL(G, \bss,\bsp)$ is finite if and only if $\BL(G', \bsphi,\bsp)$  and $\BL(Z(G)_e, \bspsi,\bsp)$  are finite.
\end{thmD}

Our final main result is to give an explicit computable value to the Brascamp--Lieb constant when it is finite.

\begin{thmE}
Suppose that $(G, \bss,\bsp)$ is a canonical Brascamp--Lieb datum of compact Lie groups.
If $\BL(G, \bss, \bsp)$ is finite, then
\begin{equation}\label{eq:BL-norm}
\BL(G, \bss, \bsp) 
= \max_{H \in \mathcal{O}(G)} \frac{ |H|_G}  {| \sigma_j(H) |_{G_j}^{1/p_j}} \,,
\end{equation}
where $\mathcal{O}(G)$ denotes the collection of open subgroups of $G$.
\end{thmE}

For compact abelian groups, a similar result may be found in \cite{BJ22}, based on earlier work of Christ and others \cite{Chr13, CDKSY24}; in \cite{BJ22} it is shown that the result extends to all finite groups.

Here is a plan of the rest of the paper. 
Section \ref{sec:2} introduces our notation and presents the necessary background material, such as the quotient integral formula and some Lie theory. 
In Section \ref{sec:BL-general-groups}, we study how the Brascamp–Lieb constant behaves when we pass to compact quotient groups and open subgroups and clarify the passage to canonical data. 
In Section \ref{sec:nilpotent}, we study the Brascamp--Lieb constant on homogeneous Lie groups, and show  that the only Brascamp--Lieb inequality is Hölder's inequality when one consider a canonical datum and $G$ is restricted to be a Heisenberg-like group.
The main content of the final section is a necessary and sufficient condition for the finiteness of the Brascamp–Lieb constant on compact Lie groups, and an explicit formula for its computation.


\section{Preliminaries}\label{sec:2}
We begin with some results on general locally compact groups.
We recall that every locally compact group $G$ has a left Haar measure, that is, a regular Borel measure that is left translation invariant, that is, $\labs xE\rabs $, the measure of $xE$,  is equal to $\labs E\rabs$, the measure of $E$, for all compact subsets $E$ of $G$; the Haar measure is unique up to a multiplicative factor.
We write $\mathrm{d}x$, $\mathrm{d}y$, etc., for the element of Haar measure.
The group $G$ is said to be \emph{unimodular} when the Haar measure is also right translation invariant.

\subsection{Quotient integral formula}
The quotient integral formula can be formulated in the setting of locally compact groups.

\begin{thm}[Quotient integral formula \cite{DE09}]\label{thm:quotient-integral-formula}
Let $G$ be a locally compact group and $N$ be a closed normal subgroup of $G$. 
The Haar measures on $G$, on $N$ and on $G/N$ may be normalised such that \begin{align}\label{eq:quotient-integral-formula}
\int_G f(x) \wrt x =\int_{{G}/{N}}\int_{N}f(xy)\wrt y\wrt x
\qquad\forall f\in C_c(G) .
\end{align}
\end{thm}

If $N$ is open in $G$, then it is natural to take the Haar measure on $N$ to be the restriction of that on $G$.
In this case, the space $G/N$ is discrete, and when it equipped with counting measure, the quotient integral formula applies.  
We call these choices the standard open subgroup normalisation.

Likewise, if $N$ is compact, it is natural to normalise its Haar measure to have total mass  $1$.
The choice of measure on $G/N$ such that the quotient integral formula holds is then called the standard compact quotient normalisation.

When a subgroup is both compact and open, these normalisations may differ, and so care is needed.


\subsection{Open and proper homomorphisms between locally compact groups}
For the proof of our first result, see \cite{CHM24}.

\begin{lem}\label{lem:good-or-bad}
Suppose that $\sigma\colon  G \to H$ is a homomorphism of locally compact groups.
If $\sigma$ is open, then $|\sigma(U)| > 0$ for all nonempty open subsets $U$ of $G$.
If $\sigma$ is not open, then $|\sigma(V)| = 0$ for all compact subsets $V$ of $G$.
Finally, $\sigma$ is open if and only if $\sigma(G_0)$ is open in $H$ for all open subgroups $G_0$ of $G$.
\end{lem}

The next three results are proved in \cite{BC25} for the abelian case, but the proofs go through in the general case with the only substantial change being the addition of the word ``normal'' in Proposition \ref{prop:new-homos-from-old}.
When $G$ is abelian, normality is not an issue.

\begin{lem}\label{lem:proper-homos}  
Suppose that $\sigma\colon  G \to H$ is a homomorphism of locally compact groups.
Then $\sigma$ is proper if and only if $\ker(\sigma)$ is compact, $\sigma(G)$ is closed in $H$, and $\sigma\colon  G \to \sigma(G)$ is open when $\sigma(G)$ carries the relative topology as a subspace of $H$.
\end{lem}

\begin{prop}\label{prop:new-homos-from-old}
Suppose that $\sigma\colon G \to H$ is a locally compact group homomorphism and $N$ is a closed normal subgroup of $G$ such that $\sigma(N)$ is normal in $H$. 
Then the restricted mapping $\rho\colon  N \to {\sigma(N)}\afterbar$ and the induced quotient mapping $\dot\sigma\colon  G/N \to H/({\sigma(N)}\afterbar)$ are also locally compact group homo\-mor\-phisms.
Moreover, $\sigma$ is proper if and only if $\rho$ and $\dot{\sigma}$ are proper, and if $\sigma$ is proper then $\sigma(G)$ is closed in $H$.
\end{prop}

We study homomorphisms $\bss\colon G\to \bsG $, where $\bsG \coloneqq G_1\times\dots \times G_J$ and $\bss = (\sigma_1, \dots, \sigma_J)$.

\begin{cor}\label{open}
The homomorphism $\bss\colon G\to \bsG $ is proper if and only if  the kernel $\ker(\bss)$ is compact, the subgroup $\bss(G)$ is closed in $\bsG $, and $\bss$ induces a homeomorphism of $G/\ker(\bss)$ onto $\bss(G)$ with the relative topology.
\end{cor}


\subsection{Some Lie  theory}
We provide a short summary of the general theory of Lie groups that we shall use.
Later we describe homogeneous Lie groups and compact Lie groups in more detail.
Much more on Lie groups may be found in texts such as \cite{Knapp02}, \cite{Sep07} or \cite{Var84}.

A Lie group is a locally compact group which, as a topological space, is a manifold.
The solution to Hilbert's fifth problem shows that the manifold and the group operations may be taken to be analytic.
The Lie algebra $\Lie{g}$ of a Lie group $G$ is the tangent space $T_e(G)$ to $G$ at the identity $e$.
Every $X$ in $\Lie{g}$ is the value at $e$ of a unique left-invariant vector field $X\afterarrow $, and we define the commutator of elements of $\Lie{g}$ to be the unique element $[X,Y]$ of $\Lie{g}$ such that $ [X,Y]\afterarrow = [ X \afterarrow, Y \afterarrow ]$ (where the right hand term is the commutator of left-invariant vector fields, which is again left-invariant).
All our Lie groups and algebras are taken to be finite-dimensional.

The exponential mapping $\exp\colon  \Lie{g} \to G$ is defined by the condition that 
\[
\frac{ \mathrm{d} \exp(tX) }{\mathrm{d}t} = X\afterarrow_{\exp(tX)} 
\qquad\forall t \in \R \quad\forall X \in \Lie{g} ;
\]
that is, $t \mapsto \exp(tX)$ is an integral curve for the vector field $X \afterarrow$.
For all Lie groups $G$, there is a neighbourhood $U$ of $0$ in $\Lie{g}$ such that the restriction of $\exp$ to $U$ is a diffeomorphism onto its image in $G$.

Each homomorphism $\sigma\colon G \to H$ of Lie groups induces a homomorphism $\mathrm{d}\sigma\colon  \Lie{g} \to \Lie{h}$ of the corresponding Lie algebras.  
Conversely, every homomorphism $\mathrm{d}\sigma\colon  \Lie{g} \to \Lie{h}$ of Lie algebras induces a local homomorphism $\sigma$ of the corresponding connected Lie groups $G$ and $H$, that is, a mapping from $U$, a neighbourhood of the identity in $G$ to $H$ such that $\sigma(xy) = \sigma(x)\sigma(y)$ and $\sigma(x^{-1}) = \sigma(x)^{-1}$ when all the terms are defined.
However, $\sigma$ may not extend to a homomorphism defined on all $G$ unless $G$ is simply connected.

A connected Lie group $G$ has a connected simply connected covering group $G\aftertilde$, which is a Lie group with the same Lie algebra as $G$, and the topological covering map $\pi\colon  G\aftertilde \to G$ is a homomorphism.
Each homomorphism $\sigma\colon  G \to H$ induces a homomorphism $\tilde\sigma\colon  G \aftertilde \to H \aftertilde$.

\begin{defn}
Let $\Lie{g}$ be a Lie algebra.
A (Lie) subalgebra of $\Lie{g}$ is a linear subspace $\Lie{h}$ of $\Lie{g}$ such
that $[X,Y]\in\Lie{h}$ for all $X, Y\in\Lie{h}$.
An ideal of $\Lie{g}$ is a subalgebra $\Lie{h}$ of $\Lie{g}$ such that  $[X,Y]\in\Lie{h}$ for all $X\in\Lie{g}$ and $Y\in\Lie{h}$.
The commutator subalgebra $[\Lie{g}, \Lie{g}]$ of $\Lie{g}$ is the smallest subalgebra that contains all commutators $[X,Y]$, where $X, Y \in \Lie{g}$.
The centre $\Lie{z}(\Lie{g})$ of $\Lie{g}$ is the linear subspace of all $X \in \Lie{g}$ such that $[X,Y]=0$ for all $Y \in \Lie{g}$.
The direct sum of Lie algebras $\Lie{g}$ and $\Lie{h}$, written $\Lie{g} \oplus \Lie{h}$, is the vector space $\Lie{g} \oplus \Lie{h}$ equipped with the Lie bracket 
\[
[ (X,Y), (X', Y')] = ([ X,X'], [Y,  Y'] )
\qquad\forall X, X' \in \Lie{g} \quad\forall Y, Y' \in \Lie{h}.
\]
\end{defn}

One of the main aims of Lie theory is to explain the correspondence between connected subgroups of a Lie group $G$ and subalgebras of its Lie algebra $\Lie{g}$.
In general, this correspondence is bijective, but the subgroups need not be closed.

\begin{defn}\label{def:Killing}
Let $\Lie{g}$ be a Lie algebra, and define the bilinear form $B$ on $\Lie{g}$ by
\[
B(X,Y) = \trace(\ad(X) \ad(Y))
\qquad\forall X, Y \in \Lie{g} ;
\]
$B$ is called the Cartan--Killing form of $\Lie{g}$.
\end{defn}

An inner product on $T_e(G)$, that is, on $\Lie{g}$, induces a left-invariant Riemannian metric on $G$ by left translations. 
It is natural to use this inner product to define an orthonormal basis on the cotangent bundle $T^*(G)$ and hence to determine a left-invariant volume form on $G$, which gives us a Haar measure.
Likewise we may use it to determine a translation-invariant measure on $\Lie{g}$.

\section{The Brascamp--Lieb constant on locally compact groups}\label{sec:BL-general-groups}

We now consider Brascamp--Lieb inequalities on  locally compact groups. 
Recall that $\bss$ denotes a homomorphism $(\sigma_1, \dots, \sigma_J)$ from $G$ to $\bsG  \coloneqq  G_1 \times \dots \times G_J$.

\begin{lem}\label{lem:bss-must-be-proper}
Let $(G, \bss,\bsp)$ be a Brascamp--Lieb datum such that  $\BL(G, \bss,\bsp)$ is finite.
Then $\bss\colon  G \to \bsG $ is proper and either  $\sigma_j$ is open or $p_j = \infty$.
When $p_j = \infty$, we may replace $\sigma_j$ with the canonical projection of $G$ onto the quotient group $G / \ker(\sigma_j)$, so that $\sigma_j$ becomes open, without changing the Brascamp--Lieb constant.
\end{lem}

\begin{proof}
This is proved for locally compact abelian groups in \cite{BC25}; the arguments do not rely on commutativity, and so work in our more general context.
\end{proof}

In light of this result, we may and shall assume that all $\sigma_j$ are open in what follows.

\begin{cor}\label{cor:kernel-compact}
Suppose that $(G, \bss,\bsp)$ is a Brascamp--Lieb datum.
If $\bigcap_j \ker(\sigma_j)$ is a noncompact subgroup of $G$, then $\BL(G, \bss,\bsp)$ is infinite. 
\end{cor}

By Proposition \ref{prop:new-homos-from-old}, if  $\sigma\colon  G \to H$ is a  homomorphism and $N$ is a closed normal subgroup of $G$, then $\sigma$ induces homomorphisms $\sigma|_N \colon  N \to \sigma(N) \afterbar$ and $\dot\sigma\colon  G/N \to H/ \sigma(N)\afterbar$. 

\begin{thm}\label{thm:breaking}
Suppose that  $(G, \bss,\bsp)$ is a Brascamp--Lieb datum of locally compact groups,  $N$ is a closed normal subgroup of $G$, and each $\sigma_j(N)$ is closed in $G_j$.
Then, when the Haar measures of $G$, $N$ and $G/N$, and of each $G_j$, $\sigma_j(N)$ and $G/\sigma_j(N)$ are normalised so that the quotient integral formula of Theorem \ref{thm:quotient-integral-formula} holds,
\begin{align}\label{break}
\BL(G, \bss, \bsp)\leq \BL(N, \bss|_{N}, \bsp) \BL(G/N, \bsds, \bsp).       
\end{align}
\end{thm}

\begin{proof}
If the right hand side of \eqref{break} is infinite, there is nothing to prove.
If it is finite, then both factors are finite, which implies that both $\bss|_N$ and $\bsds$ are proper, so $\bss$ is proper.

Take $f_j\in C_c(G_j)$.
From the quotient integral formula, 
\allowdisplaybreaks
\begin{align*}
&\int_{G} \prod_{j} f_j ( \sigma_j(x) )\wrt x\\
&= \int_{G/N}\int_{N} \prod_{j} f_j( \sigma_j(xy) ) \wrt y \wrt {\dot x} \\
&\leq\BL(N, \bss|_{N}, \bsp) \int_{G/N} \prod_{j}
\biggl( \int_{\sigma_j(N)} |f_j( \sigma_j(x) y_j ) |^{p_j} \wrt y_j \biggr) ^{1/p_j} \wrt {\dot x}  \\
&\leq \BL(N, \bss|_{N}, \bsp) \BL(G/N, \dot{\bss}, \bsp)
\prod_{j}
\biggl( \int_{G_j/\sigma_j(N)} \int_{\sigma_j(N)} |f_j( x_j y_j )|^{p_j} \wrt y_j  \,\wrt {\dot x}_j \biggr) ^{1/p_j} \\
&= \BL(N, \bss|_{N}, \bsp) \BL(G/N, \dot{\bss}, \bsp)
\prod_{j} \biggl( \int_{G_j} |f_j( x_j )|^{p_j} \wrt x_j \biggr) ^{1/p_j} ,
\end{align*}
which proves \eqref{break}.
On the third line, we applied the Brascamp--Lieb inequality with homo\-mor\-phisms $\sigma_j|_{N}$ and inputs $y_j \mapsto f_j(\sigma_j(x) y_j)$ on $\sigma_j(N)$, on the fourth line, we applied the Brascamp--Lieb inequality with homomorphisms $\dot\sigma_j$ and inputs
\[
\dot x_j \mapsto \biggl(\int_{\sigma_j(H)} |f_j(x_j y_j)|^{p_j} \wrt y_j\biggr)^{1/p_j}
\]
on $G_j/\sigma_j(N)$, and on the last line, we used the quotient integral formula.
\end{proof}

\begin{cor}\label{cor:breaking}
Suppose  that  $(G^1, \bss^1,\bsp)$ and $(G^2, \bss^2,\bsp)$ are Brascamp--Lieb data of locally compact groups.
Then, when the Haar measures of $G^1 \times G^1$ and of each $G^1_j \times G^2_j$  are the products of the Haar measure of the relevant factors, 
\begin{align}\label{break-product}
\BL(G^1\times G^2 , \bss^1 \otimes \bss^2, \bsp) =  \BL(G^1, \bss^1, \bsp) \BL(G^2, \bss^2, \bsp).       
\end{align}
\end{cor}

\begin{proof}
The previous theorem implies that 
\begin{align*}
\BL(G^1\times G^2 , \bss^1 \otimes \bss^2, \bsp) 
\leq  \BL(G^1, \bss^1, \bsp) \BL(G^2, \bss^2, \bsp).       
\end{align*}
The opposite inequality follows by testing on tensor products of near extremisers for the two factors.
\end{proof}

Next we consider restriction to open subgroups and factoring out compact subgroups.
The next two lemmas are  proved for locally compact abelian groups in \cite{BC25}; the arguments do not rely on commutativity, and so work in our more general context.

\begin{lem}\label{lem:BL-open-sbgp}
Suppose that $(G, \bss,\bsp)$ is a  Brascamp--Lieb datum, $G_0$ is an open subgroup of $G$, and $H_j$ is an open subgroup of $G_j$ such that $\sigma_j(G_0) \subseteq H_j$ for each $j$.
Then the induced maps ${\sigma}^\circ_j\colon  G_0 \to H_j$, defined to be ${\sigma_j}|_{G_0}$ with codomain $H_j$, are homomorphisms. 
Further, when $G_0$ and $H_j$ are equipped with the standard open subgroup Haar measures,
\[
\BL(G_0, \bss^\circ, \bsp)
\leq \BL(G, \bss, \bsp).
\]
The inequalities are equalities when $G_0 = G$.
\end{lem}

\begin{lem}\label{lem:BL-cpct-qtnt}
Suppose that $(G, \bss,\bsp)$ is a Brascamp--Lieb datum, and that $N$ is a compact normal subgroup of $G$ such that each $\sigma_j(N)$ is a compact normal subgroup of $G_j$.
Then the induced quotient mappings $\dot\sigma_j\colon  G/N \to G_j/\sigma_j(N)$ are homomorphisms,
and $\bsds$ is proper if $\bss$ is proper.
Further, when the Haar measures on $G/N$ and $G_j/\sigma_j(N)$ are given the standard compact quotient normalisations,
\[
\BL(G/N, \dot{\bss}, \bsp) \leq \BL(G, \bss, \bsp).
\]
If $N \subseteq \ker(\sigma_j)$ for all $j$, then equality holds.
\end{lem}

With the aid of these two lemmas, we are able to reduce the problem of finding the Brascamp--Lieb constant  to that of finding the constant with canonical data, which we now define.

\begin{defn}
A canonical Brascamp--Lieb datum is a Brascamp--Lieb datum $(G, \bss,\bsp)$ where $G$ is a closed subgroup of the product group $\bsG \coloneqq G_1\times\dots \times G_J$ and each $\sigma_j$ is the restriction to $G$ of the canonical projection of $\bsG$ onto $G_j$.
\end{defn}

\begin{lem}
Suppose that $(G, \bss, \bsp)$ is a Brascamp--Lieb datum such that $\BL(G, \bss, \bsp)$ is finite, and let $N = \ker(\bss)$ and $G_j^\circ = \sigma_j(G)$.
Then $N$ is compact and each $G_j^\circ$ is open in $G_j$.  
The homomorphisms $\sigma_j$ induce homomorphisms $\tilde\sigma_j\colon  G/N \to G_j^\circ$, given by $\tilde\sigma_j(xN) = \sigma_j(x)$, and $\bsts(G)$ is a closed subgroup of $\bsGc$.
Further, $\tilde\sigma_j$ is the restriction to $G/N$ of the canonical projection $\pi_j$ of $\bsGc \coloneqq  G_1^\circ \times \dots \times G_J^\circ$ to $G_j^\circ$.
Finally, 
\[
\BL(G,\bss,\bsp) = \BL(G/N, \bsts, \bsp).
\]
\end{lem}

\begin{proof}
By Proposition \ref{prop:new-homos-from-old}, $\bsts$ is proper, so $\bss(G)$ is closed in $\bsG$ and $\bss$ is a homeomorphic isomorphism from $G$ to $\bss(G)$, equipped with the relative topology.  
Thus we may identify $G$ with $\bss(G)$, and we have a canonical datum.

The equality of the constants follows from Lemmas \ref{lem:BL-open-sbgp} and \ref{lem:BL-cpct-qtnt}.
\end{proof}

Now we point out various simplifications for canonical Brascamp--Lieb data.

First, we may and shall henceforth assume that $J > 1$.
Indeed, if $J=1$, then the homomorphism $\sigma_1$ is an isomorphism, and isomorphisms preserve Haar measures, up to a scaling factor.
Hence the case where $J=1$ is trivial.

Next we consider the possibility that one or more of the $p_j$ is equal to $\infty$ or $1$.

\begin{prop}\label{prop:pk=infty}
Suppose that $1 \leq k \leq J$ and $p_k = \infty$.
Denote by $\bsts$ and $\bstp$ the collection of homomorphisms $\sigma_1, \dots, \sigma_J$ with $\sigma_k$ omitted and the collection of indices $p_1, \dots, p_J$ with $p_k$ omitted.
Then
\[
\BL(G, \bss, \bsp)
= \BL(G, \bsts, \bstp) .
\]
\end{prop}

\begin{proof}
The inequalities 
\[
\BL(G, \bsts, \bstp) \leq \BL(G, \bss, \bsp)
\] 
are easy.  
Equality relies on the existence of sequences $(u_n)$ in $S(G_k)$ that converge to $1$ uniformly on compacta and satisfy $\norm{u_n}_{L^\infty(G_k)} \leq 1$.
\end{proof}

\begin{prop}\label{prop:pk=1}
Suppose that $1 \leq k \leq J$ and $p_k = 1$. 
Denote $\ker(\sigma_k)$ by $N$, by $\bsrho$ the collection of restricted homomorphisms $\sigma_1|_{N}, \dots, \sigma_J|_{N}$ with codomains $\sigma_j(N)$ and with $\sigma_k|_{N}$ omitted, and by $\bsq$ the collection of indices $p_1, \dots, p_J$  with $p_k$ omitted.
Then $\BL(G, \bss, \bsp)$ is finite if and only if $\sigma_j(N)$ is open in $G_j$ for all $j$ other than $k$ and $\BL(N, \bsrho, \bsq)$ is finite.

Moreover, if $\BL(G, \bss, \bsp)$ is finite, the Haar measures on $G$, $N$, and $G/N$ are such that the quotient integral formula \eqref{eq:quotient-integral-formula} holds, the Haar measures on $G_j$, $\sigma_j(N)$, and $G_j/\sigma_j(N)$ have the standard open normalisations when $j \neq k$ and the Haar measure on $G_k$ is the push-forward of the Haar measure on $G/N$, then
\[
\BL(G, \bss, \bsp)
= \BL(N, \bsrho, \bsq).
\]
\end{prop}

\begin{proof}
We assume throughout this proof that the quotient integral formula \eqref{eq:quotient-integral-formula} holds for $G$, $N$ and $G/N$.

Suppose that $C \coloneqq  \BL(G, \bss, \bsp)$ is finite.
Take $z \in G$, and let $f_k$ be an element of an approximation to the delta distribution at $\sigma_k(z)$ in $C_c(G_k)$; for $j \neq k$, take $f_j \in C_c(G_k)$.
Since $\bss$ is proper, $\bigcap_j \sigma_j^{-1}(\supp(f_j))$ is a compact subset of $G$.
Then \eqref{eq:quotient-integral-formula} shows that
\begin{equation}\label{eq:L1-reduction-1}
\begin{aligned}
\int_{G} \prod_{j} f_j(\sigma_j(x)) \wrt{x}
&= \int_{G/N} f_k(\sigma_k(x)) \int_{N} \prod_{j\neq k} f_j(\sigma_j(xy)) \wrt{y} \wrt{\dot{x}} \\
&\to \int_{N} \prod_{j\neq k} f_j(\sigma_j(z) \rho_j(y)) \wrt{y} ,
\end{aligned}
\end{equation}
whence
\begin{equation}\label{eq:L1-reduction-2}
\bigglabs \int_{N} \prod_{j\neq k} f_j(\sigma_j(z)\rho_j(y)) \wrt{y} \biggrabs
\leq C \prod_{j\neq k} \norm{ f_j }_{L^{p_j}(G_j)} 
\end{equation}
for all $f_j \in C_c(G_j)$.
This holds for all $z \in G$, and in particular when $z = e$, and so 
\begin{equation}\label{eq:L1-reduction-3}
\bigglabs \int_{N} \prod_{j\neq k} f_j(\rho_j(y)) \wrt{y} \biggrabs
\leq C \prod_{j\neq k} \norm{ f_j }_{L^{p_j}(G_j)} 
\qquad\forall f_j \in C_c(G_j).
\end{equation}

Let $U$ be a relatively compact open set in $N$, and for $j$ other than $k$, choose $f_j$ that take the value $1$ on the relatively compact set $\sigma_j(U)$.
Then the left hand side of \eqref{eq:L1-reduction-3} is bounded below by $|U|$.
If $\sigma_j(U)$ were a null set in $G_j$ for some $j$ other than $k$, then we could choose $f_j \in C_c(G_j)$, equal to $1$ on $\sigma_j(U)$, with arbitrarily small ${L^{p_j}(G_j)}$ norm, and contradict \eqref{eq:L1-reduction-3}.
Then no $\sigma_j(U)$ is a null set, so by Lemma \ref{lem:good-or-bad},  each $\sigma_j(N)$ is an open subgroup of $G_j$. 
We take the standard open normalisation for the Haar measures of $G_j$, $\sigma_j(N)$ and $G_j/\sigma_j(N)$ and consider $f_j$ supported in $\sigma_j(N)$. 
It follows that
\[
\bigglabs \int_{N} \prod_{j\neq k} f_j(\rho_j(y)) \wrt{y} \biggrabs
\leq C \prod_{j\neq k} \norm{ f_j }_{L^{p_j}(\sigma_j(N))} 
\qquad\forall f_j \in C_c(\sigma_j(N)),
\]
so $\BL(N, \bsrho, \bsq) \leq C = \BL(G, \bss, \bsp)$.

Conversely, suppose that $\sigma_j(N)$ is open in $G_j$ for all $j$ other than $k$ and that the Haar measures are normalised as enunciated.
Suppose also that $C \coloneqq  \BL(N, \bsrho, \bsq)$ is finite.
Then \eqref{eq:L1-reduction-3} holds by definition, so \eqref{eq:L1-reduction-2} holds, by the translation invariance of Haar measure, and 
\[
\begin{aligned}
\bigglabs\int_{G} \prod_{j} f_j(\sigma_j(x)) \wrt{x} \biggrabs
&\leq \int_{G/N} |f_k(\sigma_k(x))| \bigglabs \int_{N} \prod_{j\neq k} f_j(\sigma_j(xy)) \wrt{y} \biggrabs \wrt{\dot{x}} \\
&\leq C \norm{f_k}_{L^1(G_k)} \prod_{j\neq k} \norm{ f_j (\sigma_j(x) \cdot)|_{\sigma_j(N)} }_{L^{p_j}(\sigma_j(N))}\\
&\leq C \norm{f_k}_{L^1(G_k)} \prod_{j\neq k} \norm{ f_j }_{L^{p_j}(G_j)},
\end{aligned}
\]
and so $\BL(G,\bss,\bsp) \leq C = \BL(N,\bsrho,\bsq)$.
\end{proof}

Finally, we introduce the local Brascamp--Lieb constant and the Brascamp--Lieb constant for the derivative of $\bss$.

\begin{defn}
Let $(G, \bss ,\bsp )$ be a canonical Brascamp--Lieb datum of locally compact groups, and 
take $\bsU$ to be $U_1 \times \dots \times U_J$, where each $U_j$ is a relatively compact neighbourhood of the identity in $G_j$. 
Define $\BL(G,\bss ,\bsp ;\bsU)$ to be the best constant in the associated Brascamp--Lieb inequality \eqref{eq:usual} with the additional support constraint $\supp(f_j)\subseteq U_j$ for each $j$, and define the \emph{local Brascamp--Lieb constant} $\BL^{\loc}(G,\bss ,\bsp)$ to be the infimum of the constants $\BL(G,\bss ,\bsp ;\bsU)$ over all relatively compact product open sets $\bsU$.
\end{defn}

It is evident that if $\bsU \subseteq \bsV$, then $\BL(G,\bss,\bsp; \bsU) \leq \BL(G,\bss,\bsp; \bsV)$.
Thus the infimum corresponds to a limit as the sets $\bsU$ shrink towards $\{e\}$, that is, we are considering the limit along the net of open subsets of $\bsG$ ordered by reverse inclusion.

\begin{prop} \label{prop:trivial-estimate}
Suppose that  $(G, \bss ,\bsp )$ is a canonical Brascamp--Lieb datum, where $G$ is a locally compact unimodular group.
Then
\[
\BL(G,\bss ,\bsp) \geq \BL^{\loc}(G,\bss ,\bsp) .
\]
\end{prop}

\begin{proof}
This is trivial.
\end{proof}

If $(G,\bss,\bsp)$ is a Brascamp--Lieb datum of Lie groups, then $(\Lie{g}, \bsDs, \bsp)$ is a Brascamp--Lieb datum of abelian groups, as treated in \cite{BCCT08}.
In particular, our definitions give a meaning to $\BL^{\loc}(\Lie{g}, \bsDs, \bsp)$.

\begin{thm}\label{thm:G-loc-g-loc}
Let $G$ be a Lie group with Lie algebra $\Lie{g}$, and assume that the Haar measures on $G$ and all $G_j$ and the Lebesgue measures on $\Lie{g}$ and all $\Lie{g}_j$ are normalised such that the Radon--Nikodym derivatives of the exponential mappings $\exp_G$ and $\exp_{G_j}$ take the value $1$ at $0$.
If $(G,\bss ,\bsp)$ is a canonical Brascamp--Lieb datum, then 
\begin{align*}
\BL^{\loc}(G,\bss ,\bsp) = \BL^{\loc}(\Lie{g},\bsDs ,\bsp).
\end{align*}
\end{thm}

\begin{proof}
We show first that
\[
\BL^{\loc}(\Lie{g},\bsDs ,\bsp)\leq \BL^{\loc}(G,\bss ,\bsp) .
\]

Take small neighbourhoods $U_j$ of $0$ in $\Lie{g}_j$ and $V_j$ of $e$ in $G_j$ such that $\exp_{G_j} \colon U_j \to V_j $ is a diffeomorphism.
The Radon--Nikodym derivative of $\exp_{G_j}$,  denoted by $J_{G_j}$, is continuous on $\Lie{g}_j$ and takes the value $1$ at $0$.
For all functions $f_j \in C_c(\Lie{g}_j)$ such that $\supp(f_j) \subseteq U_j$, define $\tilde{f} \in C_c(G) $ by letting 
\begin{equation}\label{eq:def-approx-fn}
\tilde f_j(x)\coloneqq  
\begin{cases}
f_j (\log_{G_j} (x))  & \text{if $x \in V_j$} \\
0                      & \text{otherwise}.  
\end{cases}
\end{equation}
Clearly all functions $g_j \in C_c(G_j)$ such that $\supp(g_j) \subseteq V_j$ arise as $\tilde f_j$ for some $f_j \in C_c(\Lie{g}_j)$ such that $\supp(f_j) \subseteq U_j$. 
We claim that 
\begin{equation}\label{eq:claim}
\| \tilde{f} \|_{L^p({G_j})} 
\leq \sup_{X \in U} J_{G_j}(X)^{1/p_j} \| f_j \|_{L^{p_j} (\Lie{g}_j)}.
\end{equation}
This is obvious if $p=\infty$, while if $1\leq p<\infty$, then
\begin{align*}
\| \tilde {f} \|^p_{L^p({G_j})}
&=\int_{G_j} \labs  f(\log_{G_j} (x)) \rabs ^p\wrt x
=\int_{\Lie{g}_j}  \labs f(X) \rabs ^p J_{G_j}(X)\wrt X,
\end{align*}
by a change of variable.
Our claim \eqref{eq:claim} follows.

Without loss of generality, suppose  that the Brascamp--Lieb inputs $f_j$ are nonnegative.

Now,  $\bsV \coloneqq  V_1 \times \dots \times V_J$ is a small neighbourhood of $e$ in $\bsG$, whence $V \coloneqq  \bsV \cap G$ is a small neighbourhood of $e$ in $G$; likewise $\bsU$ is a small neighbourhood of $e$ in $\bsg$ and so $U \coloneqq  \bsU \cap \Lie{g}$ is a samll neighbourhood of $e$ in $\Lie{g}$.
If the $U_j$ are sufficiently small, then the $V_j$ are small, so $V$ is small and hence $U$ is small, so the restriction of $\exp_G$ to $U$ is a diffeomorphism, and we may write
\begin{align*}
 \int_{\Lie{g}} \prod_j f_j\lpar \mathrm{d}\sigma_j(X)\rpar  \wrt X 
&\leq \sup_{X \in U} J_{\Lie{g}}(X)^{-1}  
	\int_{\Lie{g}} \prod_j f_j\lpar \mathrm{d}\sigma_j(X)\rpar J_{\Lie{g}}(X) \wrt X .
\end{align*}
Further, the exponential maps are diffeomorphisms from $U_j$ to $V_j$ and from $U$ to $V$, and $\exp \circ \mathrm{d}\sigma_j = \sigma_j \circ \exp_G$, so
\begin{align*} 
 \int_{\Lie{g}} \prod_j f_j\lpar \mathrm{d}\sigma_j(X)\rpar J_{\Lie{g}}(X) \wrt X 
&=  \int_{G} \prod_j f_j\lpar \mathrm{d}\sigma_j\lpar \log _{G}(x)\rpar \rpar \wrt x\\
&= \int_{G} \prod_j f_j(\log _{G_{j}}\lpar \sigma_j(x)\rpar )\wrt x\\
&= \int_{G} \prod_j \tilde f_j(\sigma_j(x))\wrt x .
\end{align*}
By the definition of $\BL(G,\bss ,\bsp; \bsV)$,
\begin{align*}
 \int_{G} \prod_j \tilde f_j(\sigma_j(x)) \wrt x 
&\leq\BL(G,\bss ,\bsp; \bsV) \prod_j \norm{  \tilde f_j }_{L^{p_j}(G_j)}  \\
&=\BL(G,\bss ,\bsp: \bsV) \prod_j \sup_{X \in U} J_{G_j}(X)^{1/p_j} \prod_j \norm{  \tilde f_j }_{L^{p_j}(G_j)} .
\end{align*}
Putting everything together, it follows that
\[
\BL(\Lie{g},\bsDs ,\bsp; \bsU) 
\leq \sup_{X \in U} J_{\Lie{g}}(X)^{-1}  \prod_j \sup_{X \in U} J_{G_j}(X)^{1/p_j} \BL(G,\bss ,\bsp: \bsV) .
\]
We pass to the limit as $\bsU$ and $\bsV$ shrink, and deduce that $\BL^{\loc}(\Lie{g},\bsDs ,\bsp) \leq 
\BL^{\loc}(G,\bss ,\bsp)$.

The converse inequality is proved similarly. 
\end{proof}

\section{Examples involving homogeneous Lie groups}\label{sec:nilpotent}
We begin by reviewing some basic concepts about homogeneous groups. 
Much of the material presented here can be found in Chapter 1 of the classic treatise \cite{FS82}.

Let $G$ be a connected, simply connected nilpotent Lie group with Lie algebra $\Lie{g}$. 
Then the exponential map $\exp \colon \Lie{g} \to G$ is a global \emph{diffeomorphism}, with inverse $\log$. 
This allows us to identify $G$ with $\Lie{g}$ endowed with a group operation $(x, y) \mapsto x \diamond  y$. 
Furthermore, this group operation is a polynomial mapping and the usual Lebesgue measure on $\Lie{g}$ is a left- and right-invariant Haar measure for $G$. 

\begin{defn}
A dilation matrix on a Lie algebra is a Lie algebra homomorphism $\Lambda\colon  \Lie{g} \to \Lie{g}$ (that is, $\Lambda$ is linear and $\Lambda[X,Y] = [\Lambda X, Y] + [X, \Lambda Y]$ for all $X, Y \in \Lie{g}$) that is diagonalisable with positive eigenvalues.

Given a dilation matrix $\Lambda$ on $\Lie{g}$,  we fix a basis $\{X_1, \dots, X_n\}$ of $\Lie{g}$ consisting of eigenvectors of $\Lambda$, such that $X_i$ has eigenvalue $d_i$ and the $d_i$ increase with $i$. 
Let $\delta_t \coloneqq  \exp(\Lambda \log t)$ for all $t > 0$; then each $\delta_t$ is a Lie algebra automorphism (that is, $\delta_t$ is linear and $\delta_t[X,Y] = [\delta_t X, \delta_t Y ]$ for all $X, Y \in \Lie{g}$) and $\delta_{st} = \delta_s \circ \delta_t$ for all $s, t >0$.
Further,  $\delta_t X_j = t^{d_j} X_j$ for all $i$. 
We say that $(\delta_t)_{t > 0}$ is a \emph{family of dilations on $\Lie{g}$}.

A \emph{homogeneous group} is a connected, simply connected Lie group $G$ whose Lie algebra $ \Lie{g}$ is endowed with a family of dilations $(\delta_t)_{t > 0}$. 
The group $G$ is nilpotent and the maps $\exp \circ \,\delta_t\, \circ \log  \colon G \to G$ are group automorphisms, called \emph{dilations of $G$} and denoted by $\delta_t$. 
The \emph{homogeneous dimension} of a homogeneous group $G$ is the quantity
\begin{equation*}
    Q \coloneqq  \trace(\Lambda) = \sum_{i} d_j,
\end{equation*}
where  the $d_i $ are the positive eigenvalues of  the dilation matrix $\Lambda$, listed with multiplicity. 

A \emph{homogeneous norm} on the homogeneous Lie group $G$ is a continuous map $||\,\cdot\,||_{G} \colon G \to [0, \infty)$ which is $C^{\infty}$ on $G \setminus \{0\}$ and satisfies: 
\begin{enumerate}
\item $\|x^{-1} \|_{G} = \|x\|_{G}$ and $\|\delta_tx\|_{G} = \|x\|_{G}$ for all $x \in G$ and $t > 0$;
\item $\|x\|_{G} = 0$ if and only if $x = 0$.
\end{enumerate}
\end{defn}

The Haar measure on a homogeneous Lie group $G$ has the scaling property
\begin{equation*}
 t^{-Q} \int_{G} f \circ \delta_{t^{-1}}(x) \wrt x = \int_{G} f(x) \wrt x 
\end{equation*}
for all $f \in L^1(G)$ and all $t > 0$.

It is easy to see that a homogeneous Lie group admits at least one homogeneous norm and we  fix one such norm $\|\,\cdot\,\|_{G}$ on $G$.
Then there exists a constant $C_G \geq 1$ such that 
\begin{equation}\label{eq: quasi triangle}
   \|x y\|_{G} \leq C_G\bigl(\|x\|_{G} + \|y\|_{G}\bigr) 
\qquad\forall x, y \in G.
\end{equation}

We now focus on the behaviour of the Brascamp--Lieb constant on homogeneous groups.
Here it is natural to restrict our attention to homomorphisms $\sigma_j \colon  G \to G_j$ that intertwine the dilations on $G$ and on $G_j$, that is
\begin{equation}\label{eq:def-homo-datum}
\sigma_j (\delta^G_t(x)) = \delta^{G_j}_t (\sigma_j (x)) 
\qquad\forall x \in G \quad\forall t > 0.
\end{equation}
In the canonical case, this amounts to asking that the dilations $\delta^{G_1}_t \otimes \dots \otimes \delta^{G_J}_t$ of $\bsG$ preserve the subgroup $G$ and agree on $G$  with the dilations $\delta^{G}_t$ of $G$ (here the notation should be obvious).
A scaling argument like that used in \cite{BCCT08} shows that, if $\BL(G, \bss,\bsp)$ is finite, that 
\begin{align}\label{homodimcond}
Q=\sum_j \frac{1}{p_j} Q_j  .
\end{align}
We use the expression ``a Brascamp--Lieb datum of homogeneous groups'' to mean a datum where $G$ and all $G_j$ are homogeneous groups and \eqref{eq:def-homo-datum} holds.

A scaling argument also underlies the next result.

\begin{thm}\label{thm:G-loc-G}
Suppose that $(G,\bss ,\bsp)$ is a canonical Brascamp--Lieb datum of homogeneous groups.
Denote by $Q$ and $Q_j$ the homogeneous dimensions of $G$ and $G_j$ respectively.
If the scaling condition \eqref{homodimcond} holds, then $\BL^{\loc}(G,\bss ,\bsp)= \BL(G,\bss ,\bsp)$.
\end{thm}

\begin{proof}
Let $\bsU$ be a product neighbourhood of $e$ in $\bsG$, that is, $\bsU = U_1 \times \dots \times U_J$, where each $U_j$ is a relatively compact neighbourhood of $e$ in $G_j$.
By definition and Proposition \ref{prop:trivial-estimate}, it suffices to show that $\BL(G,\bss ,\bsp ;\bsU)
\geq     \BL(G,\bss ,\bsp)$.
Now
\[ 
\BL(G,\bss ,\bsp) 
= \sup_{f_j\in L^{p_j}(G_j) \setminus \{0\}} 
	\frac{\biglabs \int_G \prod_j f_j (\sigma_j(x)) \wrt x \bigrabs}{\prod_j \norm{f_j}_{L^{p_j}(G_j)}}.
\]
A density argument shows that it suffices to consider $C_c(G_j)$ inputs $f_j$ in the supremum.
Fix such inputs $f_j$; then there exists $\lambda < 1$ such that $\delta_{\lambda} (\supp(f_j)) \subseteq U_j$ for each $j$.
Let $\tau_j$ be the contractive automorphism $\delta_{\lambda}$ on $G_j$, and consider the inputs $f_j \circ \tau_j$.
Then 
\[ 
\begin{aligned}
 \dfrac{\biglabs \int_G \prod_j f_j (\tau_j(\sigma_j(x))) \wrt x \bigrabs}{\prod_j \norm{f_j \circ \tau_j}_{L^{p_j}(G_j)}} 
&=  \dfrac{\lambda^{Q} \biglabs \int_G  \prod_j f_j (\tau_j(\sigma_j(x))) \wrt x \bigrabs}{\lambda^{\sum{Q_j/p_j}} \prod_j \norm{f_j }_{L^{p_j}(G_j)}} 
= \dfrac{ \biglabs \int_G \prod_j f_j (\sigma_j(x)) \wrt x \bigrabs}{\prod_j \norm{f_j}_{L^{p_j}(G_j)}} \,.
\end{aligned}
\]
Hence 
\[
\BL(G,\bss ,\bsp ; \bsU) \geq     \BL(G,\bss ,\bsp),
\]
and the proof is complete.
\end{proof}

\begin{cor} 
Suppose that $(G,\bss ,\bsp)$ is a canonical Brascamp--Lieb datum of homogeneous groups.
Denote by $Q$ and $Q_j$ the homogeneous dimensions of $G$ and $G_j$ respectively.
If the scaling condition \eqref{homodimcond} holds, then $\BL(G,\bss ,\bsp)= \BL(\Lie{g}, \bsDs ,\bsp)$.
\end{cor}

\begin{proof}
This follows by applying Theorems \ref{thm:G-loc-g-loc} and \ref{thm:G-loc-G} to $\BL(G, \bss,\bsp)$ and $\BL(\Lie{g}, \bsDs,\bsp)$ (using the dilations $\delta_t$ of the Lie algebra, not scalar multiplications).
\end{proof}

We conclude this section with the observation that there are no Brascamp--Lieb inequalities other than multilinear Hölder inequalities where the group $G$ is the Heisenberg group or a group like the Heisenberg group, in a sense that we explain shortly.

\begin{defn}
Let $n\in\mathbb{N}$. 
The \emph{Heisenberg group} $\Hn$ is defined to be the group with underlying space $\C^n\times\R$, equipped with the group law: 
\begin{align*}
(z,t)  \diamond  (z',t') \coloneqq  \lpar z+z', t + t' + \tfrac{1}{2} \Im(\bar z \cdot z') \rpar .
\end{align*}
The centre $Z(\Hn)$ of $\Hn$ is the subgroup $\{ (0,t) : t \in \R \}$.
\end{defn}

\begin{lem}\label{lem:kernel-homom-Heisenberg}  
Suppose that $N$ is a normal subgroup of the Heisenberg group $\Hn$.
Then either $N \subseteq (0, \R)$ or  $N \supseteq (0, \R)$.
\end{lem}

\begin{proof}
If $(z,t) \in N$, then $(z,t)^{-1} \diamond (z',t')^{-1} \diamond (z,t) \diamond (z',t') \in N$ for all $(z', t') \in \Hn$.
A straightforward calculation shows that
\[ 
(z,t)^{-1} \diamond(z',t')^{-1} \diamond  (z,t) \diamond (z' , t') = (0, \Im(\bar 
z \cdot z') ).
\]
If $N \not\subseteq (0, \R)$, then there exists $(z,t) \in N$ such that $z \neq 0$;
in this case, the commutator $(z,t)^{-1} \diamond(z',t')^{-1} \diamond  (z,t) \diamond (z',t') $ varies over $(0, \R)$ as $(z',t')$ varies over $\Hn$, and $N \supseteq (0, \R)$.
Otherwise  $N \subseteq (0, \R)$, as required.
\end{proof}

A homogeneous group $G$ is Heisenberg-like if it has a normal subgroup $N$, isomorphic to $\R$, with the property that every normal subgroup of $G$ is either contained in $N$ or contains $N$.
The groups of $n \times n$ upper triangular unipotent matrices are Heisenberg-like when $n \geq 3$.

Before our next theorem, we recall a simple variant of Kronecker's approximation theorem (see, e.g., \cite{Boh34, HW54}).

\begin{thm}\label{approKron}
Given finitely many real numbers $\alpha_1,\cdots,\alpha_J$ and arbitrarily small $\varepsilon>0$, there exists a sequence $\{t_m\}_m$ such that $t_m\to\infty$ when $m\to\infty$ and, for all $m$ and $j$, there exists an integer $k_{j,m}$ such that 
\[
\labs  { t_m}-  \alpha_j k_{j,m} \rabs <\varepsilon
\qquad\forall1\leq j\leq J \quad\forall m \in \N.
\]
\end{thm}

\begin{prop}
The only Brascamp--Lieb inequalities in which $G$ is a Heisenberg-like group are multilinear Hölder inequalities.
\end{prop}

\begin{proof}
The subgroup $N$ is isomorphic to  $\R$.
We write $(0, t)$ for the element of $N$ that corresponds to $t \in \R$.
Then the closed subgroups of $N$ are all of the form $(0, \Z \alpha)$ for some $(0, \alpha) \in N$.

We may suppose that each $\sigma_j\colon  G \to G_j$ is a surjective open homomorphism of $G$ onto a group $G_j$.
The kernel of $\sigma$ is a closed normal subgroup of $G$, and $G$ is Heisenberg-like, so $\ker\sigma_j$ is either $\{e\}$, or of the form $(0, \alpha_j \Z)$, or contains $(0,\R)$.
By renumbering the $\sigma_j$ if necessary, we may and shall suppose that $\ker\sigma_j = \{e\}$ when $j = 1, \dots, I$.
If there are no such $\sigma_j$, then $I = 0$ and the next two paragraphs may be skipped.

For $j \leq I$, the group  $G_j$ is isomorphic to $G$.
Now the push-forward of the Haar measure on $G$ is a positive multiple of the  Haar measure on $G_j$, so $f_j \in L^{p_j}(G_j)$ if and only if $f_j \circ \sigma_j \in L^{p_j}(G)$ and there is a constant $C_j$ such that $\| f_j \circ \sigma_j\|_{L^{p_j}(G)} = C_j \| f_j \|_{L^{p_j}(G_j)}$.
By Hölder's inequality, $\prod_{j \leq I} f_j \circ \sigma_j \in L^p(G)$, where $1/p = 1/p_1 + \dots + 1/p_I$, and every function in $L^p(G)$ arises in this way.
Since $\prod_j f_j \circ \sigma_j$ is integrable by assumption, $p \geq 1$.
Therefore, by the converse of Hölder's inequality, the Brascamp--Lieb inequality 
\[
\int_{G} \prod_j f_j \circ \sigma_j(x) \wrt x \leq C \prod_{j} \| f_j \|_{L^{p_j}(G_j} 
\]
holds if and only if 
\[
\bigglpar\int_{G} \Biglabs \prod_{j > I} f_j \circ \sigma_j(x) \Bigrabs^{p'} \wrt x \biggrpar^{1/p'} 
\leq C' \prod_{j\geq I} \| f_j \|_{L^{p_j}(G_j} ,
\]
for some (possible different) constant $C'$, and this is in turn equivalent to the inequality
\[
\int_{G} \bigglabs \prod_{j > I} f_j \circ \sigma_j(x) \biggrabs \wrt x
\leq C' \prod_{j > I} \bigglpar \int_{G_j} \labs f_j(x) \rabs^{p_j /p'} \wrt x \biggrpar^{p'/p_j}
\]
for all $f \in L^{p_j/p}(G_j)$; that is,  a Brascamp--Lieb inequality holds for the datum $(G, \bsrho, \bsq)$, where $(\rho_1, \dots, \rho_{J'})$ is the collection of homomorphisms $(\sigma_{I+1}, \dots, \sigma_{J})$ and $(q_1, \dots, q_{J'})$ is the collection of indices $(p_{I+1}/p', \dots, p_{J}/p')$.
(It might be that some $p_j/p' < 1$, but this is not a matter of concern.)

If $J = I$, then we have a multilinear Hölder inequality.

Otherwise $J > I$ and  there are homomorphisms $\sigma_{I+1}$, \dots, $\sigma_J$, all of whose kernels are nontrivial, and for each $j > I$, there exists $\alpha_j \in \Rplus $ such that $(0, \alpha_j \Z) \subseteq  \ker\sigma_j$.
We take a nonempty relatively compact open set $U$ in the Heisenberg group, and choose nonnegative functions $f_j \in L^{p_j/p'}(G_j)$ such that $f_j(y) = 1$ for all $y \in ( 0,  [-\epsilon,\epsilon])  \diamond U$.

In light of Theorem \ref{approKron}, there is an infinite sequence of elements $(0,t_m)$ of $(0, \R)$ such that $t_m \to \infty$ as $m \to \infty$ and, for all $m$ and $j$, there exists $k_{j,m} \in \Z$ such that $\labs t'_m \rabs < \epsilon$, where $t_m' = t_m - \alpha_j k_{j,m}$.  
By passing to a subsequence if necessary, we may suppose that the sets $(0,t_n)^{-1} \diamond U$ are disjoint.
Since $(0, \alpha_j k) \in \ker(\sigma_j )$ for all integers $k$,
\[
\begin{aligned}
\int_{G} \prod_j  f_j \circ \sigma_j(x) \wrt x 
&\geq \sum_{m} \int_{(0,t_m)^{-1} \diamond U} \prod_j  f_j \circ \sigma_j(x) \wrt x \\
&= \sum_{m} \int_{U} \prod_j  f_j \circ \sigma_j((0,t_m) \diamond  x) \wrt x \\
&= \sum_{m} \int_{U} \prod_j  f_j \circ \sigma_j((0,t_m' ) \diamond  x) \wrt x \\
&= \sum_{m} \int_{U} 1 \wrt x \\
&= \infty.
\end{aligned}
\]
This shows that no Brascamp--Lieb inequality can hold in this case.

Consequently, there cannot be any $\sigma_j$ whose kernel is not $\{e\}$, and the only possible Brascamp--Lieb inequality is a multilinear Hölder inequality.
\end{proof}

Young's inequality for convolution for the Heisenberg group is a well known example of a Brascamp--Lieb inequality involving the Heisenberg group, but products of Heisenberg groups are needed to write this; see Example \ref{ex:Young}  for more information.

\begin{rem}
A Loomis--Whitney  inequality for the Heisenberg group was established in \cite{FP22}; however, the Heisenberg projections in that article are {not} homomorphisms.
\end{rem}

\section{The Brascamp--Lieb constant on compact Lie groups}
We begin this section with a review of the structure theory of compact Lie groups.

\begin{defn}
A Lie algebra $\Lie{g}$ is \emph{compact} if it is the Lie algebra of a compact Lie group.
A Lie algebra $\Lie{g}$ is \emph{simple} if it is not abelian and it has no ideals besides $\{0\}$ and itself; it is \emph{semisimple} provided that there are simple Lie ideals $\Lie{s}_i$ such that $\Lie{g}=\sum_{i\in I} \Lie{s}_i$.
\end{defn}

A compact Lie algebra is a direct sum of its semisimple commutator algebra $[\Lie{g}, \Lie{g}]$ and its centre $\Lie{z}(\Lie{g})$.

If the compact semisimple Lie algebra $\Lie{g}$ is a sum of simple Lie algebras $\Lie{s}_i$, then $-B$, the negative of the Killing form, is an inner product on $\Lie{g}$ relative to which the subalgebras $\Lie{s}_i$ are orthogonal, and the restriction of the Killing form of $\Lie{g}$ to $\Lie{s}_i$ is the Killing form of $\Lie{s}_i$.

Let $G$ be a compact Lie group.
The negative of the Killing form is an inner product on $\Lie{g}'$.
Since $\Lie{g} = \Lie{g}' \oplus \Lie{z}(\Lie{g})$, we may chose an arbitrary inner product on $\Lie{z}(\Lie{g})$ to obtain an inner product on $\Lie{g}$ for which $\Lie{g}'$ and $\Lie{z}(\Lie{g})$ are orthogonal.
Then the corresponding left-invariant Riemannian metric on $G$ is also right-invariant.
Further, if $\Lie{g}'$ is a sum of simple subalgebras $\Lie{s}_i$, then this inner product induces compatible left- and right-invariant Riemannian metrics on $G'$ and on connected normal subgroups of $G'$.

\begin{defn}
Let $G$ be a Lie group.
The commutator subgroup $G'$ of $G$ is the closed normal subgroup of $G$ generated by the set
\[
\{xyx^{-1}y^{-1} : x,y \in G\}.\]
\end{defn} 

For general Lie groups, $G'$ need not be closed, however the following holds.

\begin{prop}
Let $G$ be a compact connected Lie group.
Then the commutator subgroup $G'$ of $G$ is a closed  connected normal sub\-group of $G$ whose Lie algebra $\Lie{g}'$ is equal to $[\Lie{g},\Lie{g}]$.
To each simple ideal $\Lie{s}_i$ contained in $\Lie{g}'$ there is a closed  connected normal sub\-group $S_i$ of $G'$ whose Lie algebra is $\Lie{s}_i$.
\end{prop}

Let $G$ be a compact Lie group and $G_e$ be the connected component of $G$ that contains the identity.
Then $G_e$ is an open, closed and normal subgroup of $G$, and the quotient group $G/G_e$ is a finite discrete group.

Here is more on the structure of compact connected Lie groups (see \cite[Theorem 5.22]{Sep07}). 
\begin{thm}\label{cptconnected} Let $G$ be a compact connected Lie group, $G'$ be its commutator subgroup, and $Z(G)_e$ be the connected component of its centre $Z(G)$. 
Then $G = G'Z(G)_e$, $Z(G') = G'\cap Z(G)$ is a finite Abelian group, $Z(G)_e$ is a torus, and
\[
G \simeq [ G'\times Z(G)_e ]/F, 
\]
where $F$ is the group $G'\cap Z(G)_e$,  embedded in $G'\times Z(G)_e$ as
$\{( f, f^{-1}): f \in F\}$.
\end{thm}

\begin{lem}\label{lem:homo-cpct-Lie-gps}
Suppose that $\sigma\colon  G\to H$ is a homomorphism from a compact connected Lie group $G$ to a Lie group $H$.
Then $\sigma(G') \subseteq H'$.
Furthermore, if $\sigma$ is surjective, then $\sigma(Z(G)_e) \subseteq Z(\sigma(G))_e$.
\end{lem}

\begin{proof}
Note that $\sigma(G)$ is a compact, and hence closed subgroup of $H$, and connected. 

First, $G'$ is generated by commutators $x y x^{-1}y^{-1}$ and 
\[
\sigma (x y x^{-1}y^{-1}) = \sigma(x) \sigma(y) \sigma(x)^{-1} \sigma(y)^{-1}, 
\]
so $\sigma(G') \subseteq (\sigma(G))' \subseteq H'$.
If moreover $\sigma$ is surjective, $x \in Z(G)$ and $y \in G$, then 
\[
\sigma(x)\sigma(y) = \sigma(xy) = \sigma(yx) = \sigma(y) \sigma(x) ,
\]
which shows that $\sigma(x) \in Z(H)$.
\end{proof}

\begin{lem}\label{lem:ideal in ideal}
In a compact Lie algebra, an ideal of an ideal is an ideal.
Hence in a compact connected Lie group, a normal subgroup of a normal subgroup is a normal subgroup of the whole group.
\end{lem}
\begin{proof}
Suppose that  $\Lie{g} = \sum_{i \in I} \Lie{s}_i \oplus \Lie{z}$, where each $\Lie{s}_i$ is the simple Lie subalgebra of $\Lie{g}$ and $\Lie{z}$ is the center of $\Lie{g}$.
If $\Lie{h}$ is an ideal in $\Lie{g}$, then necessarily $\Lie{h} = \sum_{i \in I'} \Lie{s}_i \oplus \Lie{z}'$, where $I' \subset I$ and $\Lie{z}' \subseteq \Lie{z}$. 
An ideal $\Lie{j}$ of $\Lie{h}$ is necessarily of the form  $\sum_{i \in I''} \Lie{s}_i \oplus \Lie{z}''$, where $I'' \subseteq I$ and $\Lie{z}'' \subseteq \Lie{z}'$ and so is also an ideal in $\Lie{g}$.

The passage from algebras to groups is straightforward.
\end{proof}

By using the structure theory of compact Lie groups, we can reduce the question of the finiteness of the Brascamp--Lieb constant $\BL(G,\bss,\bsp)$ to that of the finiteness of the Brascamp--Lieb constants $\BL(G_c,\bss|_{G_c} ,\bsp)$.
We can also give a necessary and sufficient criterion for the the finiteness of the Brascamp--Lieb constant $\BL(G,\bss,\bsp)$, based on dimensions of subspaces, which is very similar to the criterion of Bennett and Jeong \cite{BJ22}.

We now repeat the statement of Theorem B for the reader's convenience.

\begin{thmB}
Suppose that  $(G, \bss,\bsp)$ is a canonical Brascamp--Lieb datum of compact Lie groups. 
Let $\phi_j$ be the restriction of $\sigma_j$ to $G_e$ with codomain $(G_j)_e$. 
Then  $\BL(G, \bss, \bsp)$ is finite if and only if $\BL(G_e, \bsphi, \bsp)$ is finite.
\end{thmB}

\begin{proof}
By Lemma \ref{lem:BL-open-sbgp}, if the Haar measure of $G_e$ is the restriction to $G_e$ of that of $G$, then
\[
\BL(G_e, \bsphi, \bsp) \leq \BL(G, \bss, \bsp),
\]
while by Lemma \ref{lem:BL-cpct-qtnt}, if the Haar measures of $G$, $G_e$ and $G/G_e$ as well as the Haar measures of $G_j$, $(G_j)_e$ and $G_j/ (G_j)_e$ are normalised to be $1$ for all $j$, then
\begin{equation}\label{eq:BL-cpct-gp}
\BL(G, \bss, \bsp) \leq \BL(G_e, \bsphi, \bsp) \BL(G/G_e , \bsds, \bsp),
\end{equation}
where $\dot\sigma_j$ is the induced homomorphism of $G/G_e$ into $G_j/(G_j)_e$.
Now $G/G_e$ and all $G_j/\sigma_j(G_e)$ are all finite groups, and so $\BL(G/G_e , \bsds, \bsp)$ is finite.

We note that the requirements of Lemma \ref{lem:BL-open-sbgp} and of Lemma \ref{lem:BL-cpct-qtnt} are incompatible, but from the point of view of finiteness of the Brascamp--Lieb constant, the normalisations of the Haar measures are irrelevant.
\end{proof}

For  Brascamp--Lieb data involving compact groups,	there is more to say.
As usual, we consider canonical data.

On a compact group $G$, $L^q(G) \subseteq L^p(G)$ when $1 \leq p \leq q \leq \infty$.
Hence if $\BL(G,\bss,\bsp)$ is finite and $q_j \geq p_j$ for all $j$, then $\BL(G,\bss,\bsq)$ is also finite.
This fact must be reflected in any necessary and sufficient condition for finiteness of the Brascamp--Lieb constant.
As we shall see, the natural generalisation of the finiteness condition \eqref{eq:codimension-condition} found by Bennett and Jeong for compact abelian Lie groups are the \emph{codimension conditions}
\begin{equation}\label{eq:codimension-condition-v2}
\sum_j \frac{1}{p_j} \lpar \dim(\sigma_j(G))  - \dim( \sigma_j (N) ) \rpar  \leq \dim(G) - \dim(N) 
\qquad\forall N \in \Lie{N}(G),
\end{equation}
where $\Lie{N}(G)$ is the collection of all closed connected normal subgroups of $G$, including $\{e\}$ but excluding $G$ itself (since the corresponding condition is trivial), and 
\begin{equation}\label{eq:codimension-condition-v3}
\sum_j \frac{1}{p_j} \lpar \dim(\mathrm{d}\sigma_j(\Lie{g}))  - \dim( \mathrm{d}\sigma_j (\Lie{n}) ) \rpar  
\leq \dim(\Lie{g}) - \dim(\Lie{n}) 
\qquad\forall \Lie{n} \in \Lie{I}(\Lie{g}),
\end{equation}
where $\Lie{I}(\Lie{g})$ is the collection of all ideals of $\Lie{g}$, including $\{0\}$ but excluding $\Lie{g}$ itself (since the corresponding condition is trivial).
Condition \eqref{eq:codimension-condition-v3} implies condition \eqref{eq:codimension-condition-v2}, since the Lie algebra of a closed normal subgroup is an ideal, but the converse is not immediate.
It follows from Theorems \ref{thm:BL-implies-codim-cond} and \ref{thm:codim-cond-implies-BL} below that in fact \eqref{eq:codimension-condition-v2} implies condition \eqref{eq:codimension-condition-v3}.

When $N$ is the subgroup $\{e\}$, condition \eqref{eq:codimension-condition-v2} implies that
\begin{equation*}
\sum_j \frac{1}{p_j} \dim(\sigma_j(G))   \leq \dim(G)  .
\end{equation*}
However, we are dealing with canonical data, so $G$ is a closed subgroup of $G_1 \times \dots \times G_J$, and 
\begin{equation*}
\dim(G)  \leq \sum_j \dim(\sigma_j(G)) .
\end{equation*}
It follows that, if all the codimension inequalities of \eqref{eq:codimension-condition-v2} are strict, we may decrease one or more of the $p_j$, maintaining the conditions $1\leq p_j \leq \infty$, until either all $p_j$ are equal to $1$, or at least one of the inequalities is an equality.
When all $p_j$ are equal to $1$, $G$ is equal to the product group $G_1 \times \dots \times G_J$ and the Brascamp--Lieb inequality becomes a consequence of the fact that the Haar measure of $G$ is the product of the Haar measures of the $G_j$.
So there is no loss of generality in considering the Brascamp--Lieb inequalities when at least one of the inequalities \eqref{eq:codimension-condition-v2} is an equality.

\begin{thm}\label{thm:BL-implies-codim-cond}
Suppose that  $(G, \bss, \bsp)$ is a canonical Brascamp--Lieb datum of compact Lie groups. 
If $\BL(G, \bss, \bsp)$ is finite, then the codimension condition \eqref{eq:codimension-condition-v3} holds.
\end{thm}

\begin{proof}
By Theorem B, we may assume that $G$ and all $G_j$ are connected.

As $G$ is a Lie group, we may equip it with a Riemannian metric; all such metrics are equivalent.
We equip the $G_j$ with the quotient metrics.

Take the indicator functions of the Riemannian balls $B_j(e_j,r)$ with centre $e_j$ and radius $r$ in $G_j$ in as the test functions $f_j$ in the Brascamp--Lieb inequality:
\begin{equation}\label{eq:BL-ineq-again}
\int_G \prod_j (f_j \circ \sigma_j)(x) \wrt x 
\leq C \prod_j \bigglpar \int_{G_j}  \labs f_j(x_j) \rabs^{p_j} \wrt x_j \biggrpar^{1/p_j} .
\end{equation}
If $\gamma$ is a curve of length $r$ in $G$, then its projection onto each of the $G_j$ is of length at most $r$, so that $\sigma_j(x) \in B_j(e_j, r)$ for all $x \in B(e_G, r)$.
Thus the left hand side of the inequality is at least $| B(e_G, r) |$, which, for small $r$, behaves like a positive multiple of $r^{\dim(G)}$.
Similarly, the right hand side behaves like a positive multiple of $\prod_j r^{\dim(G_j)/p_j}$.

If the inequality holds, then for all small $r$,
\[
r^{\dim(G)} \lesssim  r^{\sum_j \dim(G_j)/p_j},
\]
so that
\[
\sum_j \frac{1}{p_j} \dim(G_j) \leq \dim(G) .
\]

Let $N$ be a normal analytic subgroup of $G$, fix a large relatively compact open ball $\Lie{b}$ in $\Lie{n}$, and let $B$ be the relatively compact open set $\exp(\Lie{b})$ in $N$.
Then $|B(e_G,r) B|_G$, the Haar measure of the open set $B(e_G, r) B$ in $G$, behaves like $r^{\dim(\Lie{g}) - \dim(\Lie{n})}$ as $r \to 0$, while $|\sigma_j(B(e_G,r) B)|_{G_j}$, the Haar measure of the open set $\sigma_j(B(e_G, r)) \sigma_j(B)$ in $G_j$, behaves like $r^{\dim(\mathrm{d}\sigma_j(\Lie{g})) - \dim(\mathrm{d}\sigma_j(\Lie{n}))}$ as $r \to 0$.

Take $f_j$ to be the characteristic function of $\sigma_j(B(e_G,r) B)$.
If $x \in B(e_G,r) B$, then $f_j(\sigma_j(x)) = 1$ and so
\[
\begin{aligned}
r^{\dim(\Lie{g}) - \dim(\Lie{n})}
&\lesssim |B(e_G,r) B|_G \\
&\lesssim \int_G \prod_j (f_j \circ \sigma_j)(x) \wrt x \\
&\lesssim  \prod_j \bigglpar \int_{G_j}  \labs f_j(x_j) \rabs^{p_j} \wrt x_j \biggrpar^{1/p_j} \\
&\lesssim \prod_j r^{ (\dim(\mathrm{d}\sigma_j(\Lie{g})) - \dim(\mathrm{d}\sigma_j(\Lie{n}))) /p_j} ,
\end{aligned}
\] 
whence
\[
\sum_j \frac{1}{p_j} (\dim(\mathrm{d}\sigma_j(\Lie{g})) - \dim(\mathrm{d}\sigma_j(\Lie{n}) )) 
\leq \dim(\Lie{g}) - \dim(\Lie{n})   ,
\]
which gives \eqref{eq:codimension-condition-v2}.
\end{proof}

This theorem justifies the next definition.

\begin{defn}
Define the Brascamp--Lieb polytope $\Polytope$ to be the set of all 
$(\sfrac{1}{p_1}, \dots , \sfrac{1}{p_J})$ in $\R^J$ that satisfy the inequalities
\begin{equation}\label{eq:def-BL-polytope}
\begin{gathered}
\sum_j \frac{1}{p_j} \lpar \dim(\sigma_j(G))  - \dim( \sigma_j (N) ) \rpar  \leq \dim(G) - \dim(N) 
\\
0 \leq \frac{1}{p_j} \leq 1 ,
\end{gathered}
\end{equation}
where $N$ is a closed connected normal subgroup of $G$ other than $G$ and $1 \leq j \leq J$.
\end{defn}

Although there may be infinitely many closed connected normal subgroups $N$ of $G$ to consider, all the dimensions that appear in \eqref{eq:def-BL-polytope} are integers between $0$ and $d$, so $\Polytope$ is defined by finitely many inequalities, and hence is the convex hull of its finitely many extreme points, by the Weyl--Minkowski theorem.

We begin by considering the finiteness of the constant.

\begin{thm}\label{thm:codim-cond-implies-BL}
Suppose that $(G, \bss,\bsp)$ is a canonical Brascamp--Lieb datum of compact Lie groups.
If the codimension condition \eqref{eq:codimension-condition-v2} holds, then $\BL(G, \bss, \bsp)$ is finite.
If $G$ is connected and the Haar measures of $G$ and the $G_j$ are all normalised to be $1$, then 
\[
\BL(G, \bss, \bsp)=1.
\]
\end{thm}

\begin{proof}
By Theorem B, there is no loss of generality in supposing that $G$ and all $G_j$ are connected.
We may and shall suppose that the Haar measures of $G$, $G_e$ and $G/G_e$ are all normalised to be $1$.
It will suffice to show that $\BL(G,\bss,\bsp) \leq 1$, since testing on constant functions shows that $\BL(G,\bss,\bsp) \geq 1$,

We shall use induction on the dimension $d$ of $G$, and, for a fixed $d$, on the number $J$ of homomorphisms under consideration.
When $d = 0$ or $J= 1$, the result is obvious, so we assume that $d >0$ and $J > 1$ in what follows.

We suppose that the theorem is proved for all connected compact Lie groups of dimension less than $d$, and for all groups of dimension $d$ when the number of homomorphisms is at most $J-1$.

By the multilinear Riesz--Thorin theorem, it will suffice to show that for any extreme point $(\sfrac{1}{p_1},\dots , \sfrac{1}{p_J})$ of the polytope $\Polytope$, the Brascamp--Lieb constant is at most $1$. 
At every extreme point of $\Polytope$, at least two of the inequalities of \eqref{eq:def-BL-polytope} are equalities.

We consider two cases.
Suppose first that one of the codimension inequalities is an equality for some closed connected normal subgroup $N$ such that $0 < \dim N < \dim G$.
We normalise the Haar measures on $N$ and $G/N$ to be probability measures.

On the one hand, by Lemma \ref{lem:ideal in ideal}, the closed connected normal subgroups of $N$ are precisely the closed normal subgroups $N_1$ of $G$ that are subgroups of $N$.
Further, for such a subgroup $N_1$ of $N$,
\[
\sum_j \frac{1}{p_j} \lpar \dim(\sigma_j(G))  - \dim( \sigma_j (N_1) ) \rpar  \leq \dim(G) - \dim(N_1) 
\]
by \eqref{eq:codimension-condition-v2}, and
\[
\sum_j \frac{1}{p_j} \lpar \dim(\sigma_j(G))  - \dim( \sigma_j (N) ) \rpar = \dim(G) - \dim(N) 
\]
by hypothesis, whence
\[
\sum_j \frac{1}{p_j} \lpar \dim(\sigma_j(N))  - \dim( \sigma_j (N_1) ) \rpar  \leq \dim(N) - \dim(N_1) .
\]
Thus the codimension condition holds with $N$ in place of $G$ and $\bss|_N$ in place of $\bss$, and therefore $\BL(G, \bss|_N, \bsp) \leq 1$ by the inductive hypothesis.

On the other hand, the closed connected normal subgroups of $G/N$ are exactly the subgroups $N_1/N$, where $N_1$ is a closed connected normal subgroup of $G$ that contains $N$. 
Further, $\sigma_j$ induces a homomorphism $\dot\sigma_j$ from $G/N$ to $\sigma_j(G)/\sigma_j(N)$.
A similar argument shows that, for such $N_1$,
\[
\sum_j \frac{1}{p_j} \lpar \dim(\dot\sigma_j(G/N))  - \dim( \dot\sigma_j (N_1/N) ) \rpar  
\leq \dim(G/N) - \dim(N_1/N) ,
\]
and the codimension condition holds with $G/N$ in place of $G$ and $\bsds$ in place of $\bss$.
Therefore $\BL(G/N, \bsds, \bsp) \leq 1$ by the inductive hypothesis.

Since
\[
\BL(G, \bss, \bsp) \leq \BL(G, \bss|_N, \bsp) \BL(G/N, \bsds, \bsp), 
\]
it follows that $\BL(G, \bss, \bsp) \leq 1$ in this case.

Suppose now that none of the codimension inequalities is an equality for any $N$ such that  $0 < \dim(N) < \dim(G)$.
In this case, since we are dealing with an extreme point of $\Polytope$ and at least two of the defining inequalities hold, there must be at least one $j$ such that $1/p_j = 0$ or $1/p_j = 1$.
In this case, by appealing to Proposition \ref{prop:pk=infty} or Proposition \ref{prop:pk=1}, we may reduce the number of homomorphisms by $1$, and again appeal to the inductive hypothesis.
This proves the theorem.
\end{proof}

We note that Theorems \ref{thm:BL-implies-codim-cond} and \ref{thm:codim-cond-implies-BL} combine to prove Theorem C, which we recall.

\begin{thmC}
Suppose that $(G,\bss,\bsp)$ is a canonical datum of compact Lie groups.
If the codimension condition \eqref{eq:codimension-condition-gp} holds, then $\BL(G,\bss,\bsp)$ is finite, while if $\BL(G,\bss,\bsp)$ is finite, then the codimension condition \eqref{eq:codimension-condition} holds.
If $G$ is connected, the Haar measures of $G$ and the $G_j$ are all normalised to be $1$, and $\BL(G,\bss,\bsp)$ is finite, then 
\[
\BL(G, \bss, \bsp)=1.
\]
\end{thmC}

We now study the constant in more detail, bringing more of the structure into play.

\begin{lem}
Suppose that  $(G, \bss, \bsp)$ and $(G\aftertilde, \bsts, \bsp)$ are Brascamp--Lieb data of compact connected Lie groups, that $\pi\colon  G\aftertilde \to G$ and $\pi_j\colon  G_j\aftertilde  \to G_j$ are surjective homomorphisms with finite kernels, and that $\sigma_j \circ \pi = \pi_j \circ \tilde{\sigma}_j$.
Then $\BL(G, \bss, \bsp)$ is finite if and only if $\BL(G\aftertilde, \bsts, \bsp)$ is finite.
\end{lem}

\begin{proof}
Take small neighbourhoods $U_j$ of the identity in $G_j\aftertilde$.
If the $U_j$ are small enough, then the covering maps $\pi_j$ are local isomorphisms when restricted to $U_j$ and $\pi$ is a local isomorphism when restricted to a suitable subset $U$ of $G\aftertilde$ such that $\tilde\sigma_j(U) \subseteq U_j$.
It then follows that $\BL(G\aftertilde, \bsts, \bsp: \bsU)$ is finite if and only if $\BL(G, \bss, \bsp: \bspi\bsU)$ is finite, 
Since each group is compact, finitely many translates of $U_j$ cover $G_j\aftertilde$, and it follows that $\BL(G\aftertilde, \bsts, \bsp)$ is finite if and only if $\BL(G\aftertilde, \bsts, \bsp: \bsU)$ is finite.
Likewise, $\BL(G, \bss, \bsp)$ is finite if and only if $\BL(G, \bss, \bsp: \bsU)$ is finite.
\end{proof}

Let $ G\aftertilde$ be a finite covering group of a compact connected Lie group $G$.
Then there is a projection $\pi$ from $G\aftertilde$ to $G$, and every homomorphism $\sigma_j\colon G \to G_j$ gives rise to a homomorphism $\tilde\sigma_j\colon   G\aftertilde \to G_j$ given by $\tilde\sigma_j = \sigma_j \circ \pi$.

\begin{cor}\label{BL-covering}
With the covering group notation above, and the Haar measures of $G$ and $G\aftertilde$ both normalised to be $1$. 
Suppose $(G, \bss, \bsp)$ is a Brascamp--Lieb data of compact connected Lie groups, then
\[
\BL(G,\bss,\bsp) = \BL(G\aftertilde, \bsts, \bsp).
\]
\end{cor}
\begin{proof}
The previous result shows that one constant is finite if and only if the other constant is finite.
Once all Haar measures are normalised, the constants are either $1$ or $\infty$.
\end{proof}

We recall that every compact connected Lie group is a product $G' Z(G)_e$ of its semisimple commutator subgroup $G'$, which is semisimple, and of the connected component $Z(G)_e$ of the centre of $G$, which is a torus, and $G' \cap Z(G)_e$ is finite.
If $\sigma_j\colon  G \to G_j$ is a surjective open map, then $\sigma_j(G') = (\sigma_j(G))'$ and $\sigma_j(Z(G)_e = Z(\sigma_j(G))_e$, by Lemma \ref{lem:homo-cpct-Lie-gps}.
We write $\phi_j$  for the restriction of $\sigma_j$ to $G'$ with codomain $(\sigma_j(G))'$, and
$\psi_j$ for the restriction of $\sigma_j$ to $Z(G)_e$ with codomain $Z(G_j)_e$.

\begin{thmD}
Suppose that $(G, \bss,\bsp)$ is a canonical Brascamp--Lieb datum of compact Lie groups.
Then $\BL(G, \bss,\bsp)$ is finite if and only if $\BL(G', \bsphi,\bsp)$  and $\BL(Z(G)_e, \bspsi,\bsp)$  are finite.
\end{thmD}

\begin{proof}
By Theorem \ref{cptconnected}, $G$ has a finite covering group $G\aftertilde$ of the form $G' \times Z(G)_e$, and we may identify $G$ with ${G\aftertilde }/ N$, where $N$ is a finite normal subgroup of $G\aftertilde$.  
We lift the homomorphisms $\sigma_j\colon  G \to G_j$ to homomorphisms $\tilde\sigma_j\colon  G\aftertilde \to G_j$ by composing with the quotient map from $G\aftertilde$ to $G$.
We normalise the Haar measures of $G\aftertilde$ and $G$ to be $1$ and apply Corollary \ref{BL-covering}. 
Then
\[
\BL(G\aftertilde, \bsts, \bsp) =  \BL(G, \bss, \bsp) .
\]

We recall that the homomorphism $\tilde\sigma_j$ maps $G'$ into $(G_j)'$ and $Z(G)_e$ into $Z(G_j)_e$; hence $\tilde\sigma_j$ may be identified with $\phi_j \otimes \psi_j$.
On the one hand, if $\BL(G', \bsphi,\bsp)$  and $\BL(Z(G)_e, \bspsi,\bsp)$ are finite, then $\BL(G\aftertilde, \bsts, \bsp)$ is finite by Corollary \ref{cor:breaking}.

On the other hand, if $\BL(G\aftertilde, \bsts, \bsp)$ is finite, then testing on inputs of the form $f_j \otimes 1$ on $G' \times Z(G)_e$ shows that $\BL(G', \bsphi,\bsp)$ is finite, while testing on inputs of the form $1 \otimes f_j $ on $G' \times Z(G)_e$ shows that $\BL(G', \bspsi,\bsp)$ is finite.
\end{proof}

Now, by considering the extremal functions for the inequality, we compute the constant.

\begin{thmE}
Suppose that $(G, \bss,\bsp)$ is a canonical Brascamp--Lieb datum of compact Lie groups.
If $BL(G, \bss, \bsp)$ is finite, then
\begin{equation}\label{eq:BL-norm-2}
\BL(G, \bss, \bsp) 
= \max_{H \in \mathcal{O}(G)} \frac{ |H|_G}  {| \sigma_j(H) |_{G_j}^{1/p_j}} \,.
\end{equation}
\end{thmE}

\begin{proof}
The left and right hand sides of \eqref{eq:BL-norm-2} depend on the Haar measures of $G$ and all $G_j$ in the same way, so without loss of generality, we may and shall assume that the Haar measures of $G$, $G_c$ and $G/G_c$ and all $G_j$, $(G_j)_c$ and $(G_j)_c$ are normalised to be $1$.
Now from \eqref{eq:BL-cpct-gp} and the estimate of Christ \cite{Chr13}, extended to finite nonabelian groups by Bennett and Jeong \cite{BJ22},
\[
\BL(G, \bss, \bsp) 
\leq \max_{H \in \mathcal{O}(G)} \frac{ |H|_G}  { \prod_j | \sigma_j(H) |_{G_j}^{1/p_j}} \,.
\]

Let $f_j$ be the characteristic functions of  nontrivial open subgroups $H_j$ of $G_j$; then 
\[
H^* \coloneqq  \bigcap_j \sigma_j^{-1}(H_j) 
\]
is an open subgroup of $G$, $\int_G \prod_j (f_j\circ\sigma_j)(x) \wrt x  = \labs  H^* \rabs_{G}$,
and $\prod_j \| f_j \|_{L^{p_j}(G_j)}  = \prod_j |H_j |^{1/p_j} $.
Thus
\[
 \frac{ \labs  H^* \rabs_{G} }{\prod_j |H_j |^{1/p_j} }  \leq \BL(G, \bss ,\bsp).
\]
We take the $H_j$ to be $\sigma_j(H)$ where $H$ is an open subgroup of $G$ that maximises the right hand side of \eqref{eq:BL-norm-2}. 
By set theory, $H \subseteq H^*$, so $|H|_G \leq |H^*|_G$. 
Hence
\[
 \frac{ \labs  H^* \rabs_{G} }{\prod_j |H_j |^{1/p_j} } 
\leq \BL(G, \bss ,\bsp)
\leq  \frac{ \labs  H \rabs_{G} }{\prod_j |H_j |^{1/p_j} }  
\leq  \frac{ \labs  H^* \rabs_{G} }{\prod_j |H_j |^{1/p_j} } \,.
\] 
It now follows that $H = H^*$, the functions $f_j$ are extremisers for the Brascamp--Lieb inequality, and \eqref{eq:BL-norm-2} holds.
\end{proof}

The proof of this last theorem shows that, to find the right hand side of \eqref{eq:BL-norm-2}, it suffices to consider subgroups $H$ of $G$ that satisfy the condition that $H = \bigcap_j \sigma_j^{-1}(\sigma_j(H))$.

\end{document}